\documentclass[11pt,a4paper,reqno]{amsart}

\usepackage[dvipsname,usenames]{xcolor}
\usepackage[utf8]{inputenc}
\usepackage{etoolbox}

\usepackage{subcaption}

\usepackage[T1]{fontenc}
\usepackage[american]{babel}
\usepackage{amsmath,amssymb,amsthm,amscd,bm,mathtools}
\mathtoolsset{showonlyrefs} 
\usepackage{pgf,tikz}
\usepackage{mathrsfs}
\usepackage{enumerate}
\usetikzlibrary{arrows}
\usepackage{paralist}
\usepackage{dsfont}
\usepackage{xfrac}
\usepackage{comment}
\usepackage{esint}

\usepackage{textcomp}

\let\OLDthebibliography\thebibliography
\renewcommand\thebibliography[1]{
	\OLDthebibliography{#1}
	\setlength{\parskip}{3pt}
	\setlength{\itemsep}{2.5pt plus 0.5ex}
}

\usepackage[colorlinks,linkcolor={black},citecolor={black},urlcolor={black}]{hyperref}
\usepackage[margin=2.6cm,bmargin=3.2cm,tmargin=3.2cm]{geometry}

\renewcommand{\tilde}{\widetilde}

\newcommand{\KR}{\mathrm{KR}}
\newcommand{\TV}{\mathrm{TV}}

\newcommand{\BV}{\mathrm{BV}}
\newcommand{\BVKR}{\BV_{\KR}\bigl((0,1); \cM_+(\Td)\bigr)}
\newcommand{\cMT}{\cM_+\big((0,1)\times \Td \big)}
\newcommand{\WKR}{W_\KR^{1,1}((0,1); \cM_+(\Td))}
\newcommand{\Meps}{\R_+^{\cX_\eps}}
\newcommand{\Mdeps}{\R_a^{\cE_\eps}}

\newcommand{\MdT}{\cM^d(\T^d)}

\newcommand{\sa}{\mathsf{v}}

\newcommand{\sz}{\mathsf{z}}

\renewcommand{\hom}{\mathrm{hom}}

\DeclareMathOperator{\Eff}{\mathsf{Eff}}

\DeclareMathOperator{\Rep}{{\mathsf{Rep}}}

\newcommand{\cA}{\mathcal{A}}

\newcommand{\cE}{\mathcal{E}}
\newcommand{\cF}{\mathcal{F}}

\newcommand{\cI}{\mathcal{I}}

\newcommand{\cM}{\mathcal{M}}

\newcommand{\cW}{\mathcal{W}}
\newcommand{\X}{\mathcal{X}}

\newcommand{\cX}{\mathcal{X}}

\newcommand{\cC}{\mathcal {C}}
\newcommand{\cMA}{\mathcal{MA}}
\newcommand{\bMA}{\mathbb{MA}}
\newcommand{\V}{V}

\newcommand{\XQ}{\cX^Q}
\newcommand{\EQ}{\cE^Q}

\newcommand{\Prob}{\mathscr{P}}

\newcommand{\bA}{\mathbb{A}}

\newcommand{\T}{\mathbb{T}}

\newcommand{\Lip}{\mathrm{Lip}}
\newcommand{\Leb}{\mathscr{L}}

\newcommand{\cCE}{\mathcal{CE}}
\newcommand{\CE}{\mathsf{CE}}

\newcommand{\R}{\mathbb{R}}
\newcommand{\Z}{\mathbb{Z}}
\newcommand{\N}{\mathbb{N}}

\newcommand{\eps}{\varepsilon}
\newcommand{\conv}{\mathrm{conv}}

\newcommand{\tand}{\quad\text{ and }\quad}

\newcommand{\dd}{\, \mathrm{d}}

\newcommand{\abs}[1]{\lvert#1\rvert}

\newcommand{\norm}[1]{\| {#1}\|}

\newcommand{\suchthat}{\ensuremath{\ : \ }} 

\renewcommand{\hat}{\widehat}

\DeclareMathOperator{\sgn}{sgn}

\theoremstyle{plain}
\newtheorem{theorem}{Theorem}[section]

\newtheorem{lemma}[theorem]{Lemma}
\newtheorem{proposition}[theorem]{Proposition}
\newtheorem{Gladbach-Kopfer-Maas-Portinale:2023ecture}[theorem]{Conjecture}

\newtheorem{definition}[theorem]{Definition}
\newtheorem{assumption}[theorem]{Assumption}

\theoremstyle{remark}
\newtheorem{remark}[theorem]{Remark}

\newtheorem*{claim*}{Claim}
\newtheorem*{remark*}{Remark}
\newtheorem*{example*}{Example}
\newtheorem*{notation*}{Notation}
\numberwithin{equation}{section}

\definecolor{jan}{rgb}{0.0,0.2,0.5}
\definecolor{mat}{rgb}{0.0,0.5,0.3}

\def\bfm{{\pmb{m}}}

\def\bfJ{\pmb{J}}

\def\BS{\boldsymbol}

    \def\bfmu{{\BS\mu}}
\def\bfnu{{\BS\nu}}            \def\bfxi{{\BS\xi}}
            
\def\bfsigma{{\BS\sigma}}

\DeclareMathOperator{\dive}{\mathsf{div}}

\newcommand{\bW}{\mathbb{W}}

\newcommand{\Td}{\mathbb{T}^d}

\makeatletter
\def\moverlay{\mathpalette\mov@rlay}
\def\mov@rlay#1#2{\leavevmode\vtop{%
		\baselineskip\z@skip \lineskiplimit-\maxdimen
		\ialign{\hfil$\m@th#1##$\hfil\cr#2\crcr}}}
\newcommand{\charfusion}[3][\mathord]{
	#1{\ifx#1\mathop\vphantom{#2}\fi
		\mathpalette\mov@rlay{#2\cr#3}
	}
	\ifx#1\mathop\expandafter\displaylimits\fi}
\makeatother

\renewcommand{\phi}{\varphi}

\DeclareMathOperator{\Dom}{\mathsf{D}}

\title[Discrete-to-continuum limits of optimal transport with linear growth]{Discrete-to-continuum limits of optimal transport with linear growth on periodic graphs}

\author{Lorenzo Portinale}
\address{Institut f\"ur angewandte Mathematik, Universit\"at Bonn, Endenicher Allee 60, 53115 Bonn, Germany}
\email{portinale@iam.uni-bonn.de}

\author{Filippo Quattrocchi}
\address{Institute of Science and Technology Austria (ISTA),
	Am Campus 1, 3400 Klosterneuburg, Austria}
\email{filippo.quattrocchi@ista.ac.at}



\let\oldtocsection=\tocsection

\let\oldtocsubsection=\tocsubsection

\let\oldtocsubsubsection=\tocsubsubsection

\renewcommand{\tocsection}[2]{\vspace{0.3ex}\hspace{0em}\oldtocsection{#1}{#2}}
\renewcommand{\tocsubsection}[2]{\vspace{0.3ex}\hspace{1em}\oldtocsubsection{#1}{#2}}
\renewcommand{\tocsubsubsection}[2]{\hspace{2em}\oldtocsubsubsection{#1}{#2}}

\begin{document}

	\begin{abstract}
		We prove discrete-to-continuum convergence for dynamical optimal transport on $\Z^d$-periodic graphs with energy density having linear growth at infinity. This result provides an answer to a problem left open by Gladbach, Kopfer, Maas, and Portinale (Calc Var Partial Differential Equations 62(5), 2023), where the convergence behaviour of discrete boundary-value dynamical transport problems is proved under the stronger assumption of superlinear growth. Our result  extends the known literature to some important classes of examples, such as scaling limits of $1$-Wasserstein transport problems. Similarly to what happens in the quadratic case, the geometry of the graph plays a crucial role in the structure of the limit cost function, as we discuss in the final part of this work, which includes some visual representations. 
	\end{abstract}

	\maketitle

	\setcounter{tocdepth}{1}
	\section{Introduction}
	In the euclidean setting, the Benamou--Brenier \cite{Benamou-Brenier:2000} formulation of the distance on the space $\displaystyle \Prob_2(\R^d)$ known as \textit{2-Wasserstein} or \textit{Kantorovich-Rubinstein} distance is given by the minimisation problem 
	\begin{equation} \label{eq:BB}
		\bW_2 (\mu_0, \mu_1)^2  =  \inf \left\{ \int_0^1 \int_{\R^d} \frac{|\nu_t|^2}{\mu_t} \dd x \dd t \suchthat \partial_t \mu_t + \nabla \cdot \nu_t = 0, \quad \mu_t:\mu_0 \to \mu_1 \right\} \, , 
	\end{equation}
	for every $\mu_0,\mu_1 \in \Prob_2(\R^d)$. The PDE constraint is called \textit{continuity equation} (we write $(\bfmu,\bfnu) \in \CE$ when~$(\bfmu,\bfnu)$ is a solution).
	Over the years, the Benamou--Brenier formula \eqref{eq:BB} has revealed significant connections between the theory of optimal transport and different fields of mathematics, including partial differential equations \cite{Jordan-Kinderlehrer-Otto:1998}, functional inequalities \cite{otto2000generalization}, and the novel notion of Lott--Villani--Sturm's synthetic Ricci curvature bounds for metric measure spaces \cite{
		lott2007weak,LottVillani09,Sturm04,Sturm04b}.
	Inspired by the dynamical formulation \eqref{eq:BB}, in independent works, Maas \cite{Maas:2011} (in the setting of Markov chains) and Mielke \cite{Mielke:2011} (in the context of reaction-diffusion systems) introduced a notion of optimal transport in discrete settings structured as a dynamical formulation \textit{\`{a} la} Benamou--Brenier as in \eqref{eq:BB}. 
	One of the features of this discretisation procedure is the replacement of the continuity equation with a discrete counterpart: when working on a (finite) graph $(\cX,\cE)$ (resp.~vertices and edges), the discrete continuity equation reads
	\begin{align*}
		\partial_t m_t(x) + \sum_{y\sim x}J_t(x,y) = 0 
		\, , \quad 
		\forall x \in \cX 
		\, , \quad 
		\big(\text{we write } (\bfm, \bfJ) \in \cCE_\cX \big) 
	\end{align*}
	where $(m_t,J_t)$ corresponds to discrete masses and fluxes (s.t.~$J_t(x,y) = -J_t(y,x)$). Maas' proposed distance $\cW$ \cite{Maas:2011} is obtained by minimising, under the above constraint, a discrete analog of the Benamou--Brenier energy functional with reference measure $\pi \in \Prob(\cX)$ and weight function $\omega \in \R^\cE$, of the form
	\begin{align*}
		\int_0^1 \frac{1}{2} \sum_{(x,y) \in \cE} \hspace{-1mm}\frac{|j_t(x,y)|^2}{\hat{r}_t(x,y)} \omega(x,y) \dd t 
		\, , 
		\quad \text{where} \quad 
		\hat{r}_t(x,y) := \theta_{\log}(r_t(x), r_t(y)) 
		\,  , \quad 
	r_t(x):= \frac{m_t(x)}{\pi(x)}
		\, , 
	\end{align*}
	and where $\theta_{\log}(a,b) := \int_0^1 a^s b^{1-s} \dd s$ denotes the $1$-homogeneous, positive mean called \textit{logarithmic mean}.
	With this particular choice of the mean, it was proved \cite{Maas:2011,Mielke:2011} (see also \cite{CHLZ11}) that the discrete heat flow coincides with the gradient flow of the relative entropy with respect to the discrete distance $\cW$.
	In discrete settings, the equivalence between static and dynamical optimal transport breaks down, and the latter stands out in applications to evolution equations, discrete Ricci curvature,  and functional inequalities \cite{Erbar-Maas:2012,Mielke:2013,Erbar-Maas:2014, Erbar-Maas-Tetali:2015, Fathi-Maas:2016,Erbar-Henderson-Menz:2017,Erbar-Fathi:2018}.
	Subsequently, several contributions have been devoted to the study of scaling behaviour of discrete transport problems, in the setting of discrete-to-continuum approximation problems. The first convergence results  were obtained in \cite{GiMa13} for symmetric grids on a $d$-dimensional torus, and by \cite{trillos2017gromov} in a stochastic setting. In both cases, the authors obtained convergence of the discrete distances towards $\bW_2$ in the limit of the discretisation getting finer and finer. 
	
	Nonetheless, it turned out that the geometry of the graph plays a crucial role in the game. 
	A general result was obtained in \cite{GlKoMa18}, where it is proved that the convergence of discrete distances associated with finite-volume partitions with vanishing size to the $2$-Wasserstein space is substantially equivalent to an \textit{asymptotic isotropy condition} on the mesh. 
	The first complete characterisation of limits of transport costs on periodic graphs in arbitrary dimension for general energy functionals (not necessarily quadratic) was established in \cite{gladbach2020,Gladbach-Kopfer-Maas-Portinale:2023}: in such setting, the limit energy functional (more precisely, the energy density) can be explicitly characterised in terms of a cell-formula, which is a finite-dimensional constrained minimisation problem depending on the initial graph and the cost function at the discrete level. The energy functionals considered in \cite{Gladbach-Kopfer-Maas-Portinale:2023} are of the form
	\begin{align}
		\label{eq:intro_functionals}
		(\bfmu, \bfnu) \in \CE 
		\quad \mapsto \quad 
		\bA(\bfmu,\bfnu):=
		\int_{(0,1) \times \Td} 
		f
		\big( 
		\rho, 
		j 
		\big) 
		\dd \Leb^{d+1}
		+
		\int_{ (0,1) \times \Td} 
		f^\infty 
		\big(
		\rho^\perp, j^\perp 
		\big)
		\dd \bfsigma \, , 
	\end{align}
	where we used the Lebesgue decomposition 
	\begin{align*}
		\bfmu = \rho \Leb^{d+1} + \bfmu^\perp, 
		\quad
		\bfnu = j \Leb^{d+1} + \bfnu^\perp \, ,
		\tand 
		\bfmu^\perp = \rho^\perp \bfsigma \, , 
		\quad
		\bfnu^\perp = j^\perp \bfsigma \, ,
		\quad 
		\big( \bfsigma \perp \Leb^{d+1} \big)
	\end{align*}
	and where $f: \R_+ \times \R^d \to \R \cup \{+\infty\}$ is some given convex, lower semicontinuous function with \textit{at least linear growth}, i.e.~satisfying
	\begin{align}
		\label{eq:intro_atleast_linear}
		f(\rho, j)
		\geq c |j| - C (\rho + 1)
		\, , \qquad 
		\forall \rho \in \R_+
		\tand 
		j \in \R^d
		\, ,
	\end{align} 
	whereas $f^\infty$ denotes its recession function (see \eqref{eq:def_recess} for the precise definition). The choice $f(\rho,j) := |j|^2/\rho$ corresponds to the $\bW_2$ distance. 
	At the discrete level, on a locally finite graph $(\cX, \cE)$ embedded in $\R^d$, the natural counterpart is represented by functionals of the form
	\begin{align}
		\label{eq:intro_functionals_discrete}
		(\bfm, \bfJ) \in \cCE_\cX
		\quad \mapsto \quad 
		\cA(\bfm,\bfJ):=
		\int_0^1 
		F(m,J) \dd t 
		\, , 
	\end{align} 
	for a given lower semicontinuous, convex, and local function $F$ which also has at least linear growth with respect to the second variable (see \eqref{eq: growth} for the precise definition). 
	
	The main result in \cite{Gladbach-Kopfer-Maas-Portinale:2023} is the $\Gamma$-convergence for constrained functionals as in \eqref{eq:intro_functionals_discrete}, after a suitable rescaling of the graph $\cX_\eps := \eps \cX$, $\cE_\eps := \eps \cE$, and of the energy $F_\eps$ (and associated $\cA_\eps$), in the framework  of $\Z^d$-periodic graphs. In particular, the limit energy is of the form \eqref{eq:intro_functionals}, where the energy density $f=f_\hom$ is given in terms of a \textit{cell-formula}, explicitly reading
	\begin{align}
		\label{eq:intro_cellprobd}
		f_{\hom}(\rho,j) := \inf \left\{ F(m,j) \suchthat (m,J) \in \Rep(\rho,j) \right\} \, , \qquad \rho \in \R_+ \, , \ \ j \in \R^d \, ,
	\end{align}
	where $\Rep(\rho,j)$ denotes the set of discrete \textit{representatives} of $\rho$ and $j$, given by all $\Z^d$-periodic functions $m: \cX \to \R_+$ and $\Z^d$-periodic anti-symmetric discrete vector fields $J: \cE \to \R$ satisfying
	\begin{align}	\label{eq:intro_effective}
		\sum_{x \in \cX \cap [0,1)^d} m(x) = \rho 	
		\, , \quad 
		\dive J = 0 
		\, , 
		\tand
		\Eff (J) :=  \frac12 
		\sum_{\substack{(x,y) \in \cE \\ x \in [0,1)^d} } J(x,y) (y-x) = j \, . 
	\end{align}
	The result covers several examples, both for what concerns the geometric properties of the graph (such as isotropic meshes of $\mathbb T^d$, or the simple nearest-neighbors interaction on the symmetric grid) as well as the choice of the cost functionals (including discretisation of $p$-Wasserstein distances in arbitrary dimension and flow-based models, i.e.~when $F$ -- or $f$ -- does not depend on the first variable).
	
	\begin{figure}
		\centering
		\includegraphics[width=.65\textwidth, trim={13mm 17mm 20mm 24mm},clip]{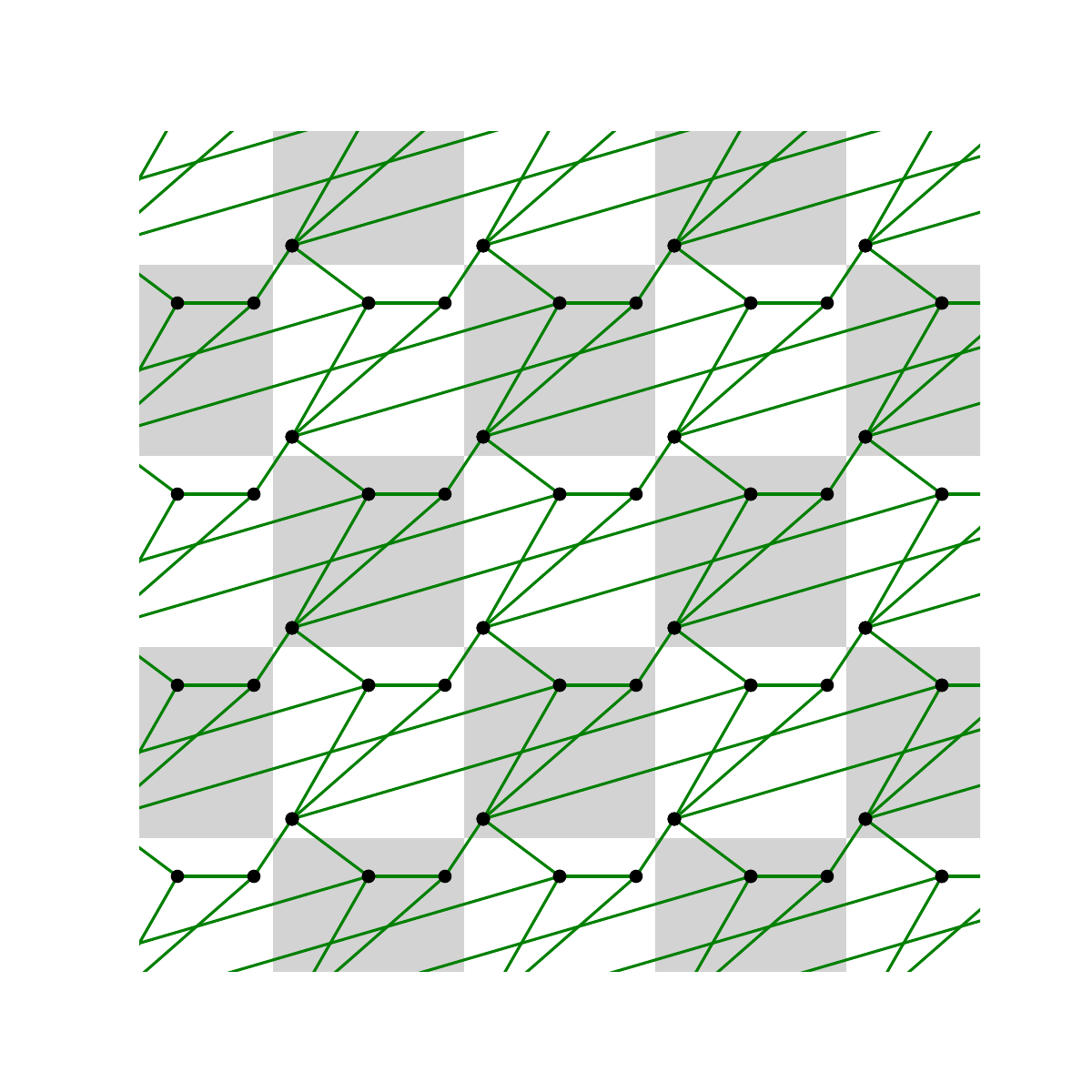}
		\caption{Example of $\Z^d$-periodic graph embedded in $\R^d$}
	\end{figure}
	
	As a consequence of this $\Gamma$-convergence result (in time-space) and a corresponding compactness result for solutions of the continuity equation with bounded energy \cite[Theorem~5.9]{Gladbach-Kopfer-Maas-Portinale:2023}, one obtains as a corollary \cite[Theorem~5.10]{Gladbach-Kopfer-Maas-Portinale:2023} that, under the stronger assumption of \textit{superlinear growth} on $F$, also the corresponding discrete boundary-value problems (i.e.~the associated squared distances, in the case of the quadratic Wasserstein problems) $\Gamma$-converge to the corresponding continuous one, namely $\cMA_\eps \stackrel{\Gamma}{\to} \bMA_{\hom}$ (with respect to the weak topology), where 
	\begin{align*}
		\cMA_\eps(m_0,m_1) 
		&:= 
		\inf 
		\left\{ 
		\cA_\eps(\bfm, \bfJ) 
		\suchthat
		(\bfm,\bfJ) \in \cCE_{\cX_\eps}
		\tand 
		\bfm_{t=0}=m_0
		, \, \, 
		\bfm_{t=1}=m_1
		\right\}
		\, , 
		\\
		\bMA_\hom(\mu_0,\mu_1)
		&:=
		\inf 
		\left\{ 
		\bA_\hom(\bfmu,\bfnu)
		\suchthat 
		(\bfmu, \bfnu) \in \CE	
		\tand 
		\bfmu_{t=0}=\mu_0
		, \, \, 
		\bfmu_{t=1}=\mu_1
		\right\}
		\, .
	\end{align*}
	The superlinear growth condition, at the continuous level, is a reinforcement of the condition \eqref{eq:intro_atleast_linear} and assumes the existence of a function $\theta:[0,\infty) \to [0,\infty)$ with $\lim_{t \to \infty} \frac{\theta(t)}{t} = \infty$ and a constant $C \in \R$ such that 
	\begin{align}
		f(\rho,j) \geq (\rho +1) \theta
		\left(
		\frac{\abs{j}}{\rho+1}
		\right)
		-
		C (\rho+1) 
		\, ,	\qquad 
		\forall \rho \in \R_+ \, , \ \ j \in \R^d 
		\, .
	\end{align} 
	In particular, this forces every finite energy $(\bfmu,\bfnu) \in \CE$ to satisfy $\bfnu \ll \bfmu + \Leb^{d+1}$ \cite[Remark~6.1]{Gladbach-Kopfer-Maas-Portinale:2023}, and it ensures compactness in	$\cC_{\KR} \bigl([0,1]; \cM_+(\T^d) \bigr)$ \cite[Theorem~5.9]{Gladbach-Kopfer-Maas-Portinale:2023}, i.e., with respect to the time-uniform convergence in the Kantorovich--Rubinstein norm (recall that the KR norm, on~$\T^d$, metrises weak convergence on $\cM_+(\T^d)$, see~\cite[Appendix A]{Gladbach-Kopfer-Maas-Portinale:2023}). This compactness property makes the proof of the convergence $\cMA_\eps \stackrel{\Gamma}{\to} \bMA_\hom$ an easy corollary of the convergence of the time-space energies. 
	
	Without the assumption of superlinear growth the situation is much more subtle: in particular, the lower semicontinuity of $\bMA$ obtained minimising the functional $\bA$ associated to a function $f$ satisfying only \eqref{eq:intro_atleast_linear} is not trivial. This is due to the fact that, in this framework, being a solution of $\CE$ with bounded energy only ensures bounds for $\bfmu \in \BVKR$, which does not suffice to pass to the limit in the constraint given by the boundary conditions: \emph{jumps} may occur at $t \in \{ 0,1\}$ in the limit. 
	Therefore, when the energy grows linearly (linear bounds from both below and above), the scaling behaviour of the discrete boundary-value problems $\cMA_\eps$, as well as the lower semicontinuity of $\bMA$, cannot be understood with the techniques  utilised in \cite{Gladbach-Kopfer-Maas-Portinale:2023}. The main goal of this work is, thus, to provide discrete-to-continuum results for $\cMA_\eps$ for discrete energies with linear growth, as well as for every flow-based type of energy, i.e.~$F(m,J) = F(J)$. With similar arguments, we can also show the lower semicontinuity of $\bMA$ for a general energy density $f$ under the same assumptions, see Section~\ref{sec:lsc}.
	
	\begin{theorem}[Main result]
	Assume that either $F$ satisfies the linear growth condition, i.e. 
	\begin{align*}
			F(m,J) \leq C
			\Bigg(
			1 +
			\sum_{
				\substack{
					(x,y) \in \cE \\ 
					|x_\sz|_{\infty} \leq R
				}
			}  
			|J(x,y)| 
			+ \sum_{
				\substack{
					x\in \cX\\ 
					|x_\sz|_{\infty} \leq R
			}} 
			m(x) 
			\Bigg) 
		\end{align*}
		for some constant $C<\infty$, or that $F$ does not depend on the $\rho$-variable (flow-based type). Then, as $\eps \to 0$, the discrete functionals $\cMA_\eps$ $\Gamma$-converge to the continuous functional $\bMA_\hom$ with respect to weak convergence.
	\end{theorem}
	
	The contribution of this paper is twofold. On one side, thanks to our main result, we can now include important examples, such as the $\bW_1$ distance and related approximations, see in particular Section~\ref{sec:analsys_cellprob} for some explicit computations of the cell-formula, including the equivalence between static and dynamical formulations \eqref{eq:MA-W1}, as well as some simulations. As typical in this discrete-to-continuum framework, also for $\bW_1$-type problems, the geometry of the graph plays an important role in the homogenised norm obtained in the limit, giving rise to a whole class of \textit{crystalline norms}, see Proposition~\ref{p:cryst} as well as Figure~\ref{fig:tilings}.  
	On the other hand, this work provides ideas and techniques on how to handle the presence of singularities/jumps in the framework of curves of measures which are only of bounded variation, which is of independent interest.  
	
	\smallskip
	\noindent
	\textit{Related literature}. \ 
	In their seminal work \cite{Jordan-Kinderlehrer-Otto:1998}, Jordan, Kinderlehrer and Otto showed that the heat flow in $\R^d$ can be seen as the gradient flow of the relative entropy with respect to the $2$-Wasserstein distance. In the same spirit, a discrete counterpart was proved in \cite{Maas:2011} and \cite{Mielke:2011}, independently, for the discrete heat flow and discrete relative entropy on Markov chains.
	In \cite{Forkert2020}, the authors 
	proved the evolutionary $\Gamma$-convergence of the discrete gradient-flow structures associated with finite-volume partitions and discrete Fokker--Planck equations, generalising a previous result obtained in \cite{Disser-Liero:2015} in the setting of isotropic, one-dimensional meshes. Similar results were later obtained in \cite{Hraivoronska-Tse:2023, Hraivoronska-Tse-Schlichting:2023} for the study of the
	limiting behaviour of random walks on tessellations in the diffusive limit. 
	Generalised gradient-flow structures associated to jump processes and approximation results of nonlocal and local-interaction equations have been studied in a series of works \cite{Esposito-Patacchini-Schlichting:2021, Esposito-Patacchini-Schlichting:2023,Esposito-Heinze-Schlichting:2023}. Recently, \cite{Esposito-Mikolas:2023} considered the more general setting where the graph also depends on time. 
	
	\subsection*{Acknowledgments}
	L.P. gratefully acknowledges fundings from the Deutsche Forschungsgemeinschaft (DFG, German Research Foundation) under Germany's Excellence Strategy - GZ 2047/1, Projekt-ID 390685813. F.Q. gratefully acknowledges support by the Austrian Science Fund (FWF), Project SFB F65. The authors are very grateful to Jan Maas for his useful feedback on a preliminary version of this work.

	\section{General framework: continuous and discrete transport problems}
	In this section, we first introduce the general class of problems at the continuous level we are interested in, discussing main properties and known results. We then move to the discrete, periodic framework in the spirit of \cite{Gladbach-Kopfer-Maas-Portinale:2023}, summarise the known convergence results, and discuss the open problems we want to treat in this work. In contrast with \cite{Gladbach-Kopfer-Maas-Portinale:2023}, for the sake of the exposition we restrict our analysis to the time interval $\cI := (0,1)$. Nonetheless, our main results easily extends to a general bounded, open interval $\cI \subset \R$.
	
	\subsection{The continuous setting: transport problems on the torus}
	We start by recalling the definition of solutions to the continuity equation on $\Td$.
	\begin{definition}[Continuity equation]
		\label{def:CE1}	
		A pair of measures $(\bfmu, \bfnu) \in 
		\cM_+\big((0,1) \times \T^d\big) 	
		\times 
		\cM^d\big((0,1) \times \T^d\big)$
		is said to be a solution to the continuity equation 
		\begin{align}	\label{eq:def_CE}
			\partial_t \bfmu + \nabla \cdot \bfnu = 0 
		\end{align}
		if, for all functions $\phi\in \cC_c^1\big((0,1)\times \Td\big)$, the identity
		\begin{align*}
			\int_{(0,1) \times \Td} 
			\partial_t \phi
			\dd \bfmu 
			+ \int_{(0,1) \times \Td} 
			\nabla \phi 
			\cdot \dd \bfnu 
			= 0
		\end{align*}
		holds. We use the notation $(\bfmu, \bfnu) \in \CE$. 
	\end{definition}
	
	Throughout the whole paper, we consider energy densities $f$ with the following properties.	
	
	\begin{assumption}
		\label{ass:f}
		Let 
		$f : \R_+ \times \R^d \to \R \cup \{+\infty\}$ 
		be a lower semicontinuous and convex function, 
		whose domain $\Dom(f)$ has nonempty interior. 
		We assume that there exist constants~$c > 0$ and~$C < \infty$ 
		such that the (at least) linear growth condition
		\begin{align}\label{eq:growth_f}
			f(\rho, j)
			\geq c |j| - C (\rho + 1)
		\end{align}	
		holds for all
		$\rho \in \R_+$ and $j \in \R^d$.
	\end{assumption}
	
	The corresponding \emph{recession function} 
	$f^\infty : \R_+ \times \R^d \to \R \cup \{+\infty\}$
	is defined by 
	\begin{align}
		\label{eq:def_recess}
		f^\infty(\rho,j) 
		:= 
		\lim_{t \to + \infty}
		\frac{f(\rho_0 + t \rho, j_0 + t j)}{t} \, ,
	\end{align}
	for any $(\rho_0, j_0) \in \Dom(f)$.
	It is well established that the function $f^\infty$ is lower semicontinuous, convex, and it satisfies the inequality
	\begin{align}	\label{eq:growth_finfty}
		f^\infty(\rho,j) 
		\geq 
		c|j| - C\rho,
		\qquad
		\rho \in \R_+ \, , \; j \in \R^d \, .
	\end{align}

	Let $\Leb^{d+1}$ denote the Lebesgue measure on $(0,1) \times \T^d$.
	For $\bfmu \in \cM_+\big( (0,1) \times \T^d\big)$ 
	and $\bfnu \in \cM^d\big( (0,1) \times \T^d\big)$, 
	we write their Lebesgue decompositions as 
	\begin{align*}
		\bfmu = \rho \Leb^{d+1} + \bfmu^\perp \, , \qquad
		\bfnu = j \Leb^{d+1} + \bfnu^\perp \, ,
	\end{align*}
	for some 
	$\rho \in L_+^1\big(  (0,1) \times \T^d \big)$ 
	and 
	$j \in L^1\big( (0,1) \times \T^d; \R^d\big)$.
	Given these decompositions, there always exists a measure $\bfsigma \in \cM_+\bigl( (0,1) \times \T^d\bigr)$ such that 
	\begin{align}	\label{eq:decomp_sigma}
		\bfmu^\perp = \rho^\perp \bfsigma \, , \qquad
		\bfnu^\perp = j^\perp \bfsigma \, ,
	\end{align}
	for some $\rho^\perp \in L_+^1(\bfsigma)$ and $j^\perp \in L^1(\bfsigma;\R^d)$ (take for example $\bfsigma := |\bfmu^\perp| + |\bfnu^\perp|$).
	
	\begin{definition}[Action functional]\label{def: A2}
		We define the action by
		\begin{align*}
			&\bA : 
			\cM_+\big( (0,1) \times \Td \big) 
			\times  
			\cM^d\big( (0,1) \times \T^d\big) 
			\to 
			\R \cup \{ +\infty \} \, , \\
			&\bA(\bfmu, \bfnu) 
			:= 
			\int_{(0,1) \times \Td} 
			f
			\big( 
			\rho, 
			j 
			\big) 
			\dd \Leb^{d+1}
			+
			\int_{ (0,1) \times \Td} 
			f^\infty 
			\big(
			\rho^\perp, j^\perp 
			\big)
			\dd \bfsigma \, , 
			\\
			&\bA(\bfmu) := \inf_{\bfnu}
			\left\{
			\bA(\bfmu, \bfnu) \suchthat (\bfmu, \bfnu) \in \CE
			\right\}
			\, .
		\end{align*}
	\end{definition}	
	
	\begin{remark}
		This definition does not depend on the choice of $\bfsigma$, due to the $1$-homogeneity of $f^\infty$. As $f(\rho,j) \geq - C(1+\rho)$ and $f^\infty(\rho,j) \geq - C \rho$ from \eqref{eq:growth_f} and \eqref{eq:growth_finfty}, the fact that $\bfmu\bigl((0,1) \times \Td \bigr) <\infty$ ensures that  $\bA(\bfmu, \bfnu)$ is well-defined 
		in $\R \cup \{ + \infty \}$.
	\end{remark}
	
	The natural setting to work in is the space $\BVKR$ of the curves of measures 	$\bfmu : (0,1) \to \cM_+(\Td)$ 
	such that the $\BV$-seminorm $\| \bfmu \| = \|\bfmu \|_{\BVKR}$ defined by
	\begin{align*}
		\label{eq:BV-def}
		\|\bfmu \|
		:= \sup\left\{ 	
		\int_{(0,1)} 
		\int_{\T^d}
		\partial_t \phi_t 
		\dd \mu_t		
		\dd t 
		\ : \ 
		\phi \in \cC_c^1\big((0,1); \cC^1(\Td)\big), 
		\ 
		\max_{t\in (0,1)}
		\|\phi_t\|_{\cC^1(\Td)} 
		\leq 1 
		\right\}
	\end{align*}
	is finite. Note that, by the trace theorem in BV, curves of measures in $\BVKR$ have a well defined trace at $t = 0$ and $t=1$. As shown in \cite[Lemma~3.13]{Gladbach-Kopfer-Maas-Portinale:2023}, any solution $(\bfmu, \bfnu) \in \CE$ can be disintegrated as 
	$\dd \bfmu(t,x) = \dd \mu_t(x) \dd t$ 
	for some measurable curve 
	$t \mapsto \mu_t \in \cM_+(\T^d)$
	with finite constant mass.
	If $\bA(\bfmu) < \infty$, 
	then this curve 
	belongs to $\BVKR$ and 
	\begin{align}
		\| \bfmu \|_{\BVKR}
		\leq 
		|\bfnu| \big( (0,1) \times \Td \big) \, .
	\end{align}
	
	\subsection*{Boundary conditions and lower semicontinuity}
	Define the minimal homogenised action for $\mu_0,\mu_1\in \cM_+(\Td)$ with $\mu_0(\Td) = \mu_1(\Td)$ as
	\begin{equation}
		\bMA(\mu_0,\mu_1) := \inf_{\bfmu \in \BVKR}\left\{ \bA(\bfmu)\,:\,\bfmu_{t=0} = \mu_0, \bfmu_{t=1} = \mu_1\right\} \, .
	\end{equation}
	Note that, in general, $\bMA$  may be infinite (although the measures have equal masses). Despite the lower semicontinuity property of $\bA$, the lower semicontinuity of $\bMA$ with respect to the natural weak topology of $\cM_+(\Td) \times \cM_+(\Td)$ is, in general, nontrivial. More precisely, it is a challenging question to prove (or disprove) that for any two sequences $\mu_0^n$, $\mu_1^n \in \cM_+(\Td)$, such that $\mu_i^n \to \mu_i$ weakly in $\cM_+(\Td)$ as $n \to \infty$ for $i=0,1$, the inequality
			\begin{align}	\label{eq:liminf_boundary}
				\liminf_{n \to \infty} \bMA(\mu_0^n, \mu_1^n)
				\geq \bMA(\mu_0,\mu_1) \,  .
			\end{align}
	holds. In this work, we provide a positive answer in the case when $f$ has linear growth or it is flow-based (i.e.~it does not depend on the first variable). First, we discuss the main challenges and the setup where the lower semicontinuity is already known to hold.
	
	\begin{remark}[Lack of compatible compactness]
		\label{rem:jumps}
		We know from \cite{Gladbach-Kopfer-Maas-Portinale:2023} that $\bA$ is lower semicontinuous w.r.t.~the weak topology on $\cM_+\big( (0,1) \times \Td \big) 
		\times  
		\cM^d\big( (0,1) \times \T^d\big)$. Moreover, if $(\bfmu_n, \bfnu_n)$ is a sequence of solutions to $\CE$ with uniformly bounded energy, i.e. 
		\begin{align}
			\sup_n 
			\bA(\bfmu_n, \bfnu_n) < \infty \,  ,
		\end{align}
		then $\bfmu_n$ is weakly compact and any limit $\bfmu$ belongs to $\BVKR$. 
		Nonetheless, this property does not ensure the lower semicontinuity of $\bMA$, because the convergence in $\BVKR$ does not preserve the boundary conditions (at time $t=0$ and  $t=1$). For similar issues in the setting of functionals of $\R^d$-valued curves with bounded variations and their minimisation, see e.g.~\cite{Amar-Vitali:1998}.
	\end{remark}
	
	\begin{remark}[Superlinear growth]
		\label{rem:superlinear}
		Under the strengthened assumption of superlinear growth on $f$ (with respect to the momentum variable), it is possible to prove the lower semicontinuity property \eqref{eq:liminf_boundary}, in the same way as in the proof of the discrete-to-continuum $\Gamma$-convergence of boundary-value problems of \cite[Theorem~5.10]{Gladbach-Kopfer-Maas-Portinale:2023}. More precisely, we say that $f$ is of \textit{superlinear growth} if there exists a function $\theta:[0,\infty) \to [0,\infty)$ with $\lim_{t \to \infty} \frac{\theta(t)}{t} = \infty$ and a constant $C \in \R$ such that 
		\begin{align}
			f(\rho,j) \geq (\rho +1) \theta
			\left(
			\frac{\abs{j}}{\rho+1}
			\right)
			-
			C (\rho+1) 
			\, ,	\qquad 
			\forall \rho \in \R_+ \, , \ \ j \in \R^d 
			\, .
		\end{align} 
		Arguing as in \cite[Remark~5.6]{Gladbach-Kopfer-Maas-Portinale:2023}, one shows that any function of superlinear growth must satisfy the growth condition given by Assumption~\ref{ass:f}. Moreover, in this case, the recession function satisfies $f^\infty(0,j)=+\infty$, for every $j \neq 0$. See \cite[Example~5.7 \& 5.8]{Gladbach-Kopfer-Maas-Portinale:2023} for some examples belonging to this class. By arguing similarly as in the proof of \cite[Theorem~5.9]{Gladbach-Kopfer-Maas-Portinale:2023}, assuming superlinear growth one can show that if $(\bfmu^n, \bfnu^n) \in \CE$ is a sequence of solutions to the continuity equation with bounded energy $\bA(\bfmu^n, \bfnu^n)$, then, up to a (nonrelabeled) subsequence, we have $\bfmu^n \to \bfmu$ weakly in $\cM_+((0,1)\times \Td)$ so that $\mu_t^n \to \mu_t$ weakly uniformly in $t \in (0,1)$, with limit curve $\bfmu \in \WKR$. Using this fact, it is clear that the problem of ``jumps'' in the limit explained in Remark~\ref{rem:jumps} does not occur, and the lower semicontinuity \eqref{eq:liminf_boundary} directly follows from the lower semicontinuity of $\bA$.
	\end{remark}
	\begin{remark}(Nonnegativity)\label{rem:nonnegative}
		Without loss of generality, we can assume that $f \geq 0$. Indeed, thanks to the linear growth assumption \ref{ass:f}, we can define a new function 
		\begin{align}	\label{eq:c1}
			\tilde f(\rho, j) := f(\rho,j) + C(\rho+1) \geq c|j| \geq 0
		\end{align}
		which is now nonnegative and with linear growth. Furthermore, we can compute the recess $\tilde f^\infty$ and from the definition we see that 
		\begin{align}	\label{eq:c2}
			\tilde f^\infty(\rho,j) = f^\infty(\rho,j) + C \rho \, .
		\end{align}
		Denote by $\tilde \bA$ the action functional obtained by replacing $f$ with $\tilde f$. Thanks to \eqref{eq:c1}, \eqref{eq:c2}, we have that 
		\begin{align}
			\tilde \bA(\bfmu) := 
			&\inf_{\bfnu}
			\left\{
			\tilde \bA(\bfmu, \bfnu) \suchthat (\bfmu, \bfnu) \in \CE
			\right\}
			\\
			=
			&\inf_{\bfnu}
			\left\{
			\bA(\bfmu, \bfnu) \suchthat (\bfmu, \bfnu) \in \CE
			\right\}
			+ C(\bfmu\big( (0,1) \times \T^d\big)+1)\, .
		\end{align}
		It follows that the corresponding boundary value problems are given by
		\begin{align}
			\tilde \bMA(\mu_0, \mu_1) := \bMA(\mu_0, \mu_1)  + C(\mu_0(\Td)+1) \, ,
			\quad \text{if } \mu_0(\Td) = \mu_1(\Td) \, .
		\end{align}
		Therefore, the (weak) lower semicontinuity for $\tilde \bMA$ is equivalent to that of $\bMA$.
	\end{remark}

	\subsection{The discrete framework: transport problems on periodic graphs}
	\label{sec:discrete}	
	We recall the framework of \cite{Gladbach-Kopfer-Maas-Portinale:2023}: let $(\cX, \cE)$ be a locally finite and $\Z^d$-periodic connected graph of bounded degree. We encode the set of vertices as $\cX  = \Z^d \times \V$, where $\V$ is a finite set, and we use coordinates $x = (x_\sz, x_\sa) \in \cX$.
	The set of edges $\cE \subseteq \cX \times \cX$ is symmetric and $\Z^d$-periodic, and
	we use the notation $x \sim y$ whenever $(x,y) \in \cE$.
	Let 
	$R_0 := 
	\max_{(x,y) \in \cE} 
	|  x_\sz - y_\sz|_{\infty}
	$ 
	be the maximal edge length in the supremum norm 
	$|\cdot|_{\infty}$ on $\R^d$.
	We use the notation
	$
	\XQ :=
	\{
	x \in \cX
	\, : \, x_\sz = 0
	\} 
	$ and 
	$
	\EQ := 
	\big\{
	(x, y) \in \cE	
	\, : \,  
	x_\sz = 0
	\big\}.
	$
	For a discussion concerning abstract and embedded graphs, see \cite[Remark~2.2]{Gladbach-Kopfer-Maas-Portinale:2023}.
	
	In what follows, we denote by~$\R^\cX_+$ the set of functions~$m \colon \cX \to \R_+$, and by~$\R^\cE_a$ the set of \emph{anti-symmetric} functions~$J \colon \cE \to \R$, that is, such that~$J(x,y) = -J(y,x)$. The elements of~$\R^\cE_a$ will often be called \emph{(discrete) vector fields}.
	
	\begin{assumption}[Admissible cost function]
		\label{ass:F}
		The function 
		$F : \R_+^\cX \times \R_a^\cE \to \R \cup \{+\infty\}$ 
		is assumed to have the following properties:
		\begin{enumerate}[(a)]
			\item\label{item:F0}
			$F$ is convex and lower semicontinuous.
			\item\label{item:F1}
			$F$ is \emph{local} meaning that, for some number $R_1 < \infty$, we have $F(m,J) = F(m',J')$ whenever $m, m' \in \R_+^\cX$ and $J, J' \in \R_a^\cE$ agree within a ball of radius $R_1$, i.e. 
			\begin{align*}
				m(x) &  = m'(x) 
				&& 
				\text{for all } x \in \cX 
				\text{ with } |x_\sz|_{\infty} \leq R_1 \, , \quad and
				\\
				J(x,y) & = J'(x,y)	
				&& 
				\text{for all } (x,y) \in \cE
				\text{ with } |x_\sz|_{\infty},
				|y_\sz|_{\infty} \leq R_1 \, .
			\end{align*}
			\item\label{item:F2}
			$F$ is of at least \emph{linear growth}, i.e.~there exist 
			$c > 0$ and $C < \infty$ 
			such that 
			\begin{align}
				\label{eq: growth}
				F(m,J) 
				\geq 
				c \sum_{ (x, y) \in \EQ } 
				|J(x,y)| 
				- C \Bigg(
				1 + \sum_{\substack{
						x\in \cX\\ 
						|x_\sz|_{\infty} \leq R
				}} 
				m(x) 
				\Bigg)
			\end{align}
			for any $m \in \R_+^\cX$ and $J \in \R_a^\cE$. Here, $R := \max\{R_0, R_1\}$.
			\item\label{item:F3} 
			There exist a $\Z^d$-periodic function 
			$m^\circ\in \R_+^\cX$
			and a $\Z^d$-periodic and 
			divergence-free vector field
			$J^\circ\in \R_a^\cE$ 
			such that
			\begin{align}\label{eq: int dom}
				( m^\circ, J^\circ ) \in \Dom(F)^\circ \, .
			\end{align}
		\end{enumerate}
	\end{assumption}	

\begin{remark}
	Important examples that satisfy the growth condition \eqref{eq: growth} are of the form
	\begin{align}\label{eq: Wp}
		F(m,J) 
		= \frac12 \sum_{(x,y) \in \EQ} 
		\frac{|J(x,y)|^p}{\Lambda\big(q_{xy} m(x), q_{yx} m(y)\big)^{p-1}} \, ,
	\end{align}
	where 
	$1 \leq p < \infty$, the constants
	$q_{xy}, q_{yx} > 0$ are fixed parameters defined for $(x,y)\in \EQ$, 
	and $\Lambda$ is a suitable mean.
	Functions of this type naturally appear in discretisations of Wasserstein gradient-flow structures \cite{Maas:2011,Mielke:2011,CHLZ11}, see also \cite[Remark~2.6]{Gladbach-Kopfer-Maas-Portinale:2023}.
\end{remark}

\smallskip
\noindent\emph{The rescaled graph.} \
Let $\T_\eps^d = (\eps \Z / \Z)^d$ be the discrete torus of mesh size $\eps \in 1/\N$.
We denote by 
$[\eps z]$ for $z \in \Z^d$
the corresponding equivalence classes.
Equivalently,
$\T_\eps^d = \eps \Z_\eps^d$
where
$\Z_\eps^d = \big(\Z / \tfrac1\eps \Z \big)^d$. 
The rescaled graph~$(\X_\eps, \cE_\eps)$ is defined as
\begin{align*}
	\cX_\eps := 
	\T_\eps^d \times \V
	\quad\text{and}\quad
	\cE_\eps := 
	\big\{
	\big( T_\eps^0 (x), T_\eps^0 (y) \big)
	\ : \
	(x, y) \in \cE
	\big\}
\end{align*}
where, for $\bar z \in \Z_\eps^d$,
\begin{align}
	\label{eq:def-T}
	T_\eps^{\bar z} : \cX \to \cX_\eps \, , 
	\qquad
	(z, v) 
	\mapsto 
	\big( [\eps (\bar z + z)], v \big) \, .
\end{align}
The equivalence relation $\sim$ on $\cX$ is equivalently defined on $\cX_\eps$ by means of $\cE_\eps$. Hereafter, we always assume that $\eps R_0 < \frac12$. 
For $x = \big([\eps z], v\big) \in \cX_\eps$ we write
\begin{align*}
	x_\sz := z \in \Z_\eps^d \, , \qquad 
	x_\sa := v\in \V \, .
\end{align*} 

\smallskip
\noindent\emph{The rescaled energies.} \
Let $F : \R_+^\cX \times \R_a^\cE \to \R \cup \{ + \infty\}$ be a cost function satisfying Assumption \ref{ass:F}. 
For $\eps > 0$ satisfying the conditions above, we can define a corresponding energy functional $\cF_\eps$ in the rescaled periodic setting: following \cite{Gladbach-Kopfer-Maas-Portinale:2023}, for $\bar z \in \Z_\eps^d$, each function 
$\psi : \cX_\eps \to \R$ 
induces a $\frac{1}{\eps}\Z^d$-periodic function 
\begin{align*}
	\tau_\eps^{\bar z} \psi 
	: \cX \to \R \, , 
	\qquad
	\big(\tau_\eps^{\bar z} \psi \big)(x)
	:=
	\psi\big( T_\eps^{\bar z}(x) \big)
	\quad \text{ for } 
	x \in \cX \, .
\end{align*}
Similarly, each function
$J : \cE_\eps \to \R$ 
induces a $\frac{1}{\eps}\Z^d$-periodic function
\begin{align*}
	\tau_\eps^{\bar z} J : \cE \to \R \, , \qquad  
	\big( \tau_\eps^{\bar z} J \big)(x,y)
	:=
	J
	\big( 
	T_\eps^{\bar z}(x),
	T_\eps^{\bar z}(y)
	\big)
	\quad \text{ for } 
	(x,y) \in \cE \, .
\end{align*}

\begin{definition}[Discrete energy functional]
	\label{eq:discrete-energy}
	The rescaled energy
	is defined by
	\begin{align*}
		\cF_\eps : 
		\Meps \times \Mdeps 
		&\to 
		\R \cup \{ + \infty\} \, ,	
		\qquad 
		(m, J) 
		\stackrel{\cF_\eps}{\longmapsto} 
		\sum_{z \in \Z_\eps^d}
		\eps^d
		F\bigg(
		\frac{\tau_\eps^z m}
		{\eps^d}
		,
		\frac{\tau_\eps^z J}
		{\eps^{d-1}}
		\bigg) \, .
	\end{align*}
\end{definition}

\begin{remark}
	\label{rem:well_posed_Feps}
	As observed in \cite[Remark~2.8]{Gladbach-Kopfer-Maas-Portinale:2023}, the functional
	$\cF_\eps(m,J)$ 
	is well-defined as an element in 
	$\R \cup \{ + \infty\}$.
	Indeed, the condition \eqref{eq: growth} yields
	\begin{align*}
		\cF_\eps( m, J ) 
		= 
		\sum_{z \in \Z_\eps^d}
		\eps^d
		F\bigg(
		\frac{\tau_\eps^z m}
		{\eps^d}
		,
		\frac{\tau_\eps^z J}
		{\eps^{d-1}}
		\bigg)
		& \geq 
		- C 
		\sum_{z \in \Z_\eps^d}
		\eps^d	
		\Bigg(
		1 + \sum_{\substack{
				x\in \cX\\ 
				|x_\sz|_{\infty} \leq R
		}} 
		\frac{\tau_\eps^z m(x)}{\eps^d}
		\Bigg)
		\\& \geq
		- C 
		\bigg( 
		1 + (2R + 1)^d 
		\sum_{x \in \cX_\eps} m(x)
		\bigg) 
		> - \infty \, .
	\end{align*} 
\end{remark}

\begin{definition}[Discrete continuity equation]	
	A pair $(\bfm, \bfJ)$ is said to be a solution to the discrete continuity equation if 
	$\bfm : (0,1) \to \Meps$ is continuous, 
	$\bfJ : (0,1) \to \Mdeps$ is Borel measurable, and 
	\begin{align}	\label{eq:def_discrete_CE}
		\partial_t m_t(x) + \sum_{y\sim x}J_t(x,y) = 0 
	\end{align}
	holds for all $x \in \cX_\eps$ in the sense of distributions.
	We use the notation 
	\begin{align*}
		(\bfm, \bfJ) \in \cCE_\eps \, .
	\end{align*}	
\end{definition}

\begin{remark}
	We may write \eqref{eq:def_discrete_CE} 
	as
	$
	\partial_t m_t + \dive J_t = 0
	$
	using the discrete divergence operator, given by
	\begin{align*}
		\dive J \in \R^{\cX_\eps} 
		\, , \qquad 
		\dive J(x) := \sum_{y \sim x} J(x,y)
		\, , \qquad 
		\forall J \in \R_a^{\cE_\eps} 
		\, .
	\end{align*}
\end{remark}

The proof of the next lemma can be found in \cite{Gladbach-Kopfer-Maas-Portinale:2023}.
\begin{lemma}[Mass preservation]
	\label{lem:mass-preservation}
	Let $(\bfm, \bfJ) \in\cCE_\eps$. 
	Then we have $m_s(\cX_\eps) = m_t(\cX_\eps)$ 
	for all $s, t \in (0,1)$.
\end{lemma}

We are now ready to define one of the main objects in this paper. 

\begin{definition}[Discrete action functional]		
	\label{def:A_eps}
	For any continuous function $\bfm: (0,1) \to \R_+^{\cX_\eps}$ such that $t \mapsto \sum_{x \in \cX_\eps} m_t(x) \in L^1\bigl((0,1)\bigr)$ and any Borel measurable function $\bfJ: (0,1) \to \R_a^{\cE_\eps}$, we define
	\begin{align*}
		\cA_\eps(\bfm, \bfJ) & :=	
		\int_0^1
		\cF_\eps( m_t, J_t ) 
		\dd t
		\in \R \cup \{ + \infty\} \, .	
	\end{align*}
	Furthermore, we set
	\begin{align*}
		\cA_\eps(\bfm)
		& :=
		\inf_{\bfJ} \Big\{
		\cA_\eps(\bfm, \bfJ) 
		\ : \
		(\bfm, \bfJ) \in \cCE_\eps
		\Big\} \, .
	\end{align*}			 
\end{definition}
By similar argument as in Remark~\ref{rem:well_posed_Feps}, one can show \cite[Remark~2.13]{Gladbach-Kopfer-Maas-Portinale:2023} that $\cA_\eps(\bfm, \bfJ)$ 
is well-defined as an element in 
$\R \cup \{ + \infty\}$, as a consequence of the growth condition \eqref{eq: growth}.

\begin{definition}[Discrete action functional]		
	\label{def:MA_eps}
	For any pair of boundary data $m_0$, $m_1 \in \R_+^{\cX_\eps}$, we define the associated discrete boundary value problem as
	\begin{align*}
		\cMA_\eps(m_0,m_1) & :=	
		\inf 
		\left\{ 	 
		\cA_\eps( \bfm ) 
		\suchthat 
		\bfm: (0,1) \to \R_+^{\cX_\eps} 
		\, , \quad 
		\bfm_{t=0} = m_0
		\tand 
		\bfm_{t=1} = m_1
		\right\} \, .	
	\end{align*}		 
\end{definition}

The aim of this work is to study the asymptotic behaviour of the energies $\cMA_\eps$ as $\eps \to 0$ under the assumption \eqref{eq: growth}. 

\section{Statement and proof of the main result}
In this paper we extend the $\Gamma$-convergence result for the functionals $\cMA_\eps$ towards $\bMA_\hom$, proved in \cite{Gladbach-Kopfer-Maas-Portinale:2023} for superlinear cost functionals, to two cases: under the assumption of linear growth (bound both from below and above) and when the function $F$ does not depend on $\rho$.

\begin{assumption}[Linear growth]
	\label{ass:F_lineargr}
	We say that a function 	$F : \R_+^\cX \times \R_a^\cE \to \R \cup \{+\infty\}$ 
	has \emph{linear growth} if it satisfies
	\begin{align*}
		F(m,J) \leq C
		\Bigg(
		1 +
		\sum_{
			\substack{
				(x,y) \in \cE \\ 
				|x_\sz|_{\infty} \leq R
			}
		}  
		|J(x,y)| 
		+ \sum_{
			\substack{
				x\in \cX\\ 
				|x_\sz|_{\infty} \leq R
		}} 
		m(x) 
		\Bigg) 
	\end{align*}
	for some constant $C<\infty$.
\end{assumption}

\begin{assumption}[Flow-based]
	\label{ass:F_fbased}
	We say that a function	$F : \R_+^\cX \times \R_a^\cE \to \R \cup \{+\infty\}$ is of \emph{flow-based type} if it depends only on the the second variable, i.e.~(with a slight abuse of notation) $F(m,J) = F(J)$, for some $F: \R_a^\cE \to \R \cup \{ +\infty \}$.  
\end{assumption}

Similarly, we say that $f:\R_+ \times \R^d \to \R$ is of \textit{flow-based type} if it does not depend on the $\rho$ variable, i.e.,~$f(\rho,j) = f(j)$. In this case, the problem simplifies significantly, and the dynamical variational problem described in Definition~\ref{def: A2} admits an equivalent, static formulation (see \eqref{eq:static_contA}).

\begin{remark}[Linear growth vs Lipschitz]
	While working with convex functions, to assume a linear growth condition (from above) is essentially equivalent to require Lipschitz continuity with respect to the second variable.
\end{remark}
\begin{lemma}[Lipschitz continuity] \label{lemma:lip}
	Let $f: \R_+ \times \R^d \to \R$ be a function, convex in the second variable. Let~$C > 0$. Then the following are equivalent:
	\begin{enumerate}
		\item for every $\rho \in \R_+$ and $j \in \R^d$ the inequality  $f(\rho,j) \le C(1+ \rho + \abs{j} )$ holds.
		\item for every $\rho \in \R_+$, the function $f(\rho,\cdot) : \R^d \to \R_+$ is Lipschitz continuous (uniformly in $\rho$) with constant~$C$, and the inequality $f(\rho,0) \le C(1+ \rho)$ holds.
	\end{enumerate}
\end{lemma}
In the very same spirit, one can show the analogous result at the discrete level.
\begin{lemma}[Lipschitz continuity II] \label{lemma:lip_disc}
	Let	$F : \R_+^\cX \times \R_a^\cE \to \R \cup \{+\infty\}$ be convex in the second variable. Let~$C > 0$. Then the following are equivalent:
	\begin{enumerate}
		\item $F$ is of linear growth, in the sense of Assumption~\ref{ass:F_lineargr}, with the same constant~$C$.
		\item For every $m \in  \R_+^\cX $, we have that
		\begin{align*}
			F(m,0) \leq C
			\bigg(
			1 + \sum_{
				\substack{
					x\in \cX\\ 
					|x_\sz|_{\infty} \leq R
			}} 
			m(x) 
			\bigg) \, ,
		\end{align*}
		as well as that $F$ is Lipschitz continuous with constant~$C$ in the second variable, in the sense that 
		\begin{align}
			\label{eq:Lip_F_disc}
			\left|
			F(m,J_1) - F(m,J_2)
			\right|
			\leq 
			C 
			\sum_{
				\substack{
					(x,y) \in \cE \\ 
					|x_\sz|_{\infty} \leq R
				}
			}  
			|J_1(x,y)- J_2(x,y)|
			\, ,  
		\end{align}
		for every $J_1, J_2 \in \R_a^\cE$.
	\end{enumerate}
\end{lemma}

\begin{proof}[Proof of Lemma~\ref{lemma:lip}]
	Let us assume the first condition and fix $\rho \in \R_+$ as well as $j_1,j_2 \in \R^d$. It follows from the convexity in the second variable that the function
	\begin{align*}
		\R \ni t \mapsto f(\rho, j_1 + t(j_2-j_1))
	\end{align*} 
	is convex. In particular, the inequalities
	\begin{align*}
		f(\rho, j_2) - f(\rho,j_1) 
		\le
		\frac
		{f(\rho, j_1 + t(j_2-j_1)) - f(\rho, j_1)}{t} 
		\le	 
		\frac	
		{ C \big( 1+\rho + \big| j_1+t(j_2-j_1)  \big| \big)   - f(\rho, j_1)}{t}
	\end{align*}
	hold for every $t \ge 1$. Letting $t \to \infty$, we thus find
	\begin{align*}
		f(\rho, j_2) - f(\rho,j_1) \le
		C \abs{j_2-j_1} 
	\end{align*}
	and, by arbitrariness of the arguments, the claimed Lipschitz continuity. The fact that $f(\rho,0) \le C(1+ \rho)$ trivially follows from the first condition.
	
	Conversely, if the second condition holds, we necessarily have
	\begin{align*}
		f(\rho,j) 
		\le 
		C  \abs{j} + f(\rho,0) 
		\le 
		C(1+ \rho + \abs{j}) 
		\, , 
	\end{align*}
	for every $\rho \in \R_+$ and $j \in \R^d$, which is precisely the first condition in the statement. 
\end{proof}
\begin{remark}
	It is not hard to show that if $F$ is of linear growth, then $f_\hom$ is also of linear growth (and therefore, in view of the latter lemma, it is Lipschitz in the second variable uniformly w.r.t.~the first one), see e.g. \cite{Gladbach-Maas-Portinale:2023}.
\end{remark}

Let us recall the effective energy density $f_\hom$ which describes the limit energy, which is given by a cell formula.
For given $\rho \geq 0$ and $j \in \R^d$, $f_\hom(\rho, j)$ is obtained by minimising over the unit cube the discrete energy among functions $m$ and vector fields $J$ \emph{representing} $\rho$ and $j$. More precisely, the function $f_\hom:\R_+ \times \R^d \to \R_+$ is given by
\begin{align}	\label{eq:intro_fhom}
	f_\hom(\rho,j) :=
	\inf_{m,J}
	\bigl\{
	F(m,J)
	\suchthat 
	(m,J) \in \Rep(\rho,j) 
	\bigr\} \, ,
\end{align}
where the set of \textit{representatives} $\Rep(\rho,j)$ consists of all $\Z^d$-periodic functions $m: \cX \to \R_+$ and all $\Z^d$-periodic anti-symmetric discrete vector fields $J : \cE \to \R$ satisfying
\begin{align}	\label{eq:intro_effective}
	\sum_{x \in \cX^Q} m(x) = \rho 	
	\, , \quad 
	\dive J = 0 
	\, , 
	\tand
	\Eff (J) :=  \frac12 
	\sum_{(x,y) \in \cE^Q} J(x,y) (y-x) = j \, . 
\end{align}
We denote by $\bA_\hom$ and $\bMA_\hom$ the action functionals corresponding to the choice $f=f_\hom$. In order to talk about $\Gamma$-convergence, we need to specificy which type of discrete-to-continuum topology/convergence we adopt (in the same spirit of \cite{Gladbach-Kopfer-Maas-Portinale:2023}).

\begin{definition}[Embedding]
	For $\eps > 0$ and $z \in \R^d$, let $Q_\eps^z := \eps z + [0, \eps)^d \subseteq \T^d$ 
	denote the projection of the cube of side-length $\eps$ based at $\eps z$ to the quotient~$\T^d = \R^d/\Z^d$.
	For $m \in \Meps$ and $J \in \R_a^{\cE_\eps}$, we define 
	$\iota_\eps m \in \cM_+(\Td)$ and $\iota_\eps J \in \MdT$ 
	by 	
	\begin{align}
		\label{eq: iota m}
		\iota_\eps m 
		&:= 
		\eps^{-d}	
		\sum_{x \in \cX_\eps}
		m(x)
		\Leb^d|_{Q_\eps^{x_\sz}}
		\, ,
		\\
		\label{eq: iota J}
		\iota_\eps J
		&	:= 
		\eps^{-d+1}	
		\sum_{(x,y) \in \cE_\eps}
		\frac{J(x,y)}{2}
		\bigg(\int_0^1
		\Leb^d|_{ Q_\eps^{(1-s)x_\sz + s y_\sz}}
		\dd s\bigg)
		(y_\sz - x_\sz)
		\, .
	\end{align}
	With a slight abuse of notation, for $\bfm: (0,1) \to \R_+^{\cX_\eps}$ we also write $\iota_\eps \bfm \in \cM_+((0,1) \times \Td)$ for the continuous space-time measure with time disintegration given by $t \mapsto \iota_\eps m_t$. Moreover, for a given sequence of nonnegative discrete measures $m^\eps \in \R_+^{\cX_\eps}$, we write 
	\begin{align*}
		m_\eps \to \mu \in \cM_+(\Td)
		\quad \text{weakly}
		\qquad 
		\text{iff}
		\qquad 
		\iota_\eps m^\eps \to \mu 
		\quad 
		\text{weakly in }
		\cM_+(\Td) 
		\, .
	\end{align*}
	Similarly, for $\bfm^\eps: (0,1) \to \R_+^{\cX_\eps}$ we write $\bfm^\eps \to \bfmu \in \cM_+((0,1) \times \Td)$ with an analogous meaning. Similar notation is used for (Borel, possibly discontinuous) curves of fluxes~$\bfJ : (0,1) \to \R^{\cE_\eps}_a$ and convergent sequences of (curves of) fluxes.
\end{definition}
\begin{remark}[Preservation of the continuity equation]
	The definition of such embedding for masses and fluxes ensures that solutions to the discrete continuity equation are mapped to solutions of $\CE$, cfr.~\cite[Lemma~4.9]{Gladbach-Kopfer-Maas-Portinale:2023}. 
\end{remark}

We are ready to state our main result.

\begin{theorem}[Main result] 
	\label{theo:main}
	Let $(\cX,\cE,F)$ be as described in Section~\ref{sec:discrete} and Assumption~\ref{ass:F}. Assume that $F$ is either of flow-based type (Assumption~\ref{ass:F_fbased}) or with linear growth (Assumption~\ref{ass:F_lineargr}). Then, in either case, the functionals $\cMA_\eps$ $\Gamma$-convergence to $\bMA_{\hom}$ as $\eps \rightarrow 0$ with respect to the weak-topology of $\cM_+(\Td) \times \cM_+(\Td)$. More precisely, we have:
\begin{enumerate}
	\item[$(1)$] \textbf{Liminf inequality}: For any sequences $m_0^\eps,m_1^\eps \in \cM_+(\cX_\eps)$ such that $m_i^\eps \to \mu_i$ weakly in $\cM_+(\Td)$ for~$i=~0,1$, we have that
	\begin{align}	\label{eq:liminf}
		\liminf_{\eps \rightarrow 0} 
		\cMA_\eps(m_0^\eps,m_1^\eps) 
		\geq 
		\bMA_{\hom}(\mu_0,\mu_1)	
		\, .
	\end{align}
		\item[$(2)$] \textbf{Limsup inequality}: For any $\mu_0,\mu_1 \in  \cM_+(\T^d)$, there exist sequences   $m_0^\eps, m_1^\eps \in \cM_+(\cX_\eps)$ such that $m_i^\eps \to \mu_i$ weakly in $\cM_+(\Td)$ for~$i=0,1$, and 
	\begin{align}	\label{eq:limsup}
		\limsup_{\eps \rightarrow 0} 
		\cMA_\eps(m_0^\eps, m_1^\eps) 
		\leq 
		\bMA_{\hom}(\mu_0,\mu_1)
		\, .
	\end{align}
\end{enumerate}
\end{theorem}

\begin{remark}[Convergence of the actions and superlinear regime]
	\label{rem:gammaconv}
	The $\Gamma$-convergence of the energies $\cA_\eps$ towards $\bA_\hom$ under Assumption~\ref{ass:F} is the main result of \cite[Theorem~5.4]{Gladbach-Kopfer-Maas-Portinale:2023}. Related to it, similarly as discussed in Remark~\ref{rem:superlinear}, the superlinear case \cite[Assumption~5.5]{Gladbach-Kopfer-Maas-Portinale:2023}, not included in the statement, has already been proved in \cite{Gladbach-Kopfer-Maas-Portinale:2023}, and it follows directly from the aforementioned convergence $\cA_\eps \xrightarrow{\Gamma} \bA_\hom$ and a strong compactness result which holds in such a framework, see in particular \cite[Theorem~5.9\&5.10]{Gladbach-Kopfer-Maas-Portinale:2023}. Without the superlinear growth assumption, the proof is much more involved and requires extra work and new ideas, which are the main contribution of this paper. 
\end{remark}

\begin{remark}[Compactness under linear growth from below]
	\label{rem:compactness}
	Just assuming Assumption~\ref{ass:F}, the following compactness result for sequences of bounded energy has been proved in \cite[Theorem~5.4]{Gladbach-Kopfer-Maas-Portinale:2023}: if $\bfm^\eps:(0,1) \to \R_+^{\cX_\eps}$ is such that
	\begin{align*}
		\sup_{\eps >0} \cA_\eps(\bfm^\eps) < \infty 
		\tand 
		\sup_{\eps >0}
		\bfm^\eps 
		\big(
		(0,1) \times \cX_\eps
		\big)
		< \infty 
		\, , 
	\end{align*}
	then there exists a curve $\bfmu= \mu_t(\dd x) \dd t \in \BVKR$ such that, up to a (nonrelabeled) subsequence, we have
	\begin{align*}
		\bfm^\eps \to \bfmu 
		\quad \text{weakly in }
		\cM_+
		\big(
		(0,1) \times \T^d
		\big)	
		\tand 
		m_t^\eps \to \mu_t
		\quad \text{weakly in }
		\cM_+(\Td)
		\, ,
	\end{align*}
	for a.e.~$t \in (0,1)$.
	This is going to be an important tool in the proof of our main result. 
\end{remark}

\subsection{Proof of the limsup inequality}
In this section we prove the limsup inequality in Theorem~\ref{theo:main}. This proof does not require Assumption~\ref{ass:F_lineargr} nor Assumption~\ref{ass:F_fbased}, but rather a weaker assumption, which is always satisfied under both hypotheses.
\begin{proposition}[$\Gamma$-limsup]
	Let $\mu_0, \mu_1$ be nonnegative measures on $\T^d$.
	Assume that there exists a $\Z^d$-periodic and divergence-free vector field $\bar J \in \R_a^\X$ such that 
	\begin{align}
		\label{eq:ass_limsup}
		F(m,\bar J) \le C \Biggl(1 + \sum_{\substack{x \in \mathcal X 
				\\ 
				\abs{x}_{\infty} \le R}} m(x)\Biggr) \, , \qquad m \in \R^{\X}_+ 
		\, , 
	\end{align}
	for some finite constant $C$. Then there exist two sequences $(m_0^\eps)_{\eps >0}$ and $(m_1^\eps)_{\eps > 0}$ in $\R_+^{\cX_\eps}$ such that $m_i^\eps \to \mu_i$ weakly in $\cM_+(\T^d)$ for $i=0,1$, and
	\begin{equation} \label{eq:plimsup:limsup} 
		\limsup_{\eps \to 0} 
		\cMA_\eps(m_0^\eps,m_1^\eps) \le \bMA_\hom(\mu_0,\mu_1)
		\, . 
	\end{equation}
\end{proposition}

\begin{proof}	
	We may and will assume that $\bMA_\hom(\mu_0,\mu_1) < \infty$. We also claim that it suffices to prove the statement with $\bMA(\mu_0,\mu_1)+1/k$ in place of the right-hand side of \eqref{eq:plimsup:limsup} for every $k \in \N_1$. Indeed, assume we know the existence of sequences $(m_i^{\eps,k})_\eps$ such that $m_i^{\eps,k} \to \mu_i$, and
	\begin{align*}
		\limsup_{\eps \to 0} \cMA_\eps(m_0^{\eps,k},m_1^{\eps,k}) \le \bMA(\mu_0,\mu_1)+1/k 
		\, , 
	\end{align*}
	for every $k \in \N_1$. Since $\T^d$ is compact, the weak convergence is equivalent to convergence in the Kantorovich–Rubinstein norm. Hence, for every $k$ we can find $\eps_k$ such that, when $\eps \le \eps_k$,
	\begin{align*}
		\cMA_\eps(m_0^{\eps,k},m_1^{\eps,k}) \le \bMA_\hom(\mu_0,\mu_1)+2/k 
		\tand 
		\max_{i=0,1} \norm{\iota_\eps m_i^{\eps,k}-\mu_i}_{\KR} \le 1/k \, . 
	\end{align*}
	We can also assume that $\eps_{k+1} \le \frac{\eps_k}{2}$, for every $k$. It now suffices to set
	\begin{align*}
		k_\eps := \max\{ k \in \N_1 \, : \, \eps_k  \ge  \eps \} 
		\tand 
		m_i^\eps := m_i^{\eps,k_\eps} 
		\, , 
	\end{align*}
	for every $\eps$ and $i=0,1$ to get
	\begin{align*}
		\limsup_{\eps \to 0} \cMA_\eps(m_0^\eps,m_1^\eps) \le \bMA_\hom (\mu_0,\mu_1)+ \limsup_{\eps \to 0} \frac{2}{k_\eps} 
	\end{align*}
	as well as
	\begin{align*}
		\limsup_{\eps \to 0} 
		\max_{i=0,1} \norm{\iota_\eps m_i^\eps -\mu_i}_{\KR} \le \limsup_{\eps \to 0}  \frac{1}{k_\eps} 
		\, . 
	\end{align*}
	The claim is proven, since $k_\eps \to_\eps \infty$, as it can be readily verified.
	
	Thus, let us now choose $k$ and keep it fixed. By definition of $\bMA_\hom$, there exists $\bfmu = \mu_t(\dd x) \dd t \in \BVKR$ with $\bfmu_{t=0} = \mu_0, \bfmu_{t=1} = \mu_1$ and such that
	\begin{align*}
		\bA_\hom(\bfmu) \le \bMA_\hom(\mu_0,\mu_1)+1/k
		\, .
	\end{align*}
	Recall from Remark~\ref{rem:gammaconv} that $\cA_\eps \xrightarrow{\Gamma} \bA_\hom$: in particular \cite[Theorem 5.1]{Gladbach-Kopfer-Maas-Portinale:2023}, there exists a recovery sequence $(\bfm^\eps, \bfJ^\eps) \in \cCE_{\eps}$ such that $\bfm^\eps \to \bfmu$ weakly and
	\begin{align*}
		\limsup_{\eps \to 0} \cA_\eps(\bfm^\eps, \bfJ^\eps) \le \bA_\hom(\bfmu)
		\, . 
	\end{align*}	
	We shall prove that $\norm{\iota_\eps m_t^\eps - \mu_t}_{\KR(\T^d)} \to 0$ in ($\Leb^1$-)measure or, equivalently, that
	\begin{equation} \label{eq:plimsup:l1} 
		\lim_{\eps \to 0} 
		\int_0^1 \min\left\{\norm{\iota_\eps m_t^\eps - \mu_t}_{\KR(\T^d)},  1 \right\} \dd t = 0 
		\, .
	\end{equation}
	In order to do this, assume by contradiction that there exists a subsequence such that
	\begin{align*}
		\int_0^1 \min\left\{\norm{\iota_{\eps_n} m_t^{\eps_n} - \mu_t}_{\KR(\T^d)}, 1 \right\} \dd t > \delta \, , \qquad n \in \N 
		\, , 
	\end{align*}
	for some $\delta > 0$. Up to possibly extracting a further subsequence, it can be easily checked that the hypotheses of \cite[Theorem 5.4]{Gladbach-Kopfer-Maas-Portinale:2023} are satisfied (cfr.~Remark~\ref{rem:compactness}), and thus there exists a further, not relabeled, subsequence, such that, for almost every $t \in (0,1)$, $m_t^{\eps_n} \to \mu_t$ weakly and thus $\norm{\iota_{\eps_n}m_t^{\eps_n} - \mu_t}_{\KR(\T^d)} \to 0$ since $\T^d$ is compact. The dominated convergence theorem yields an absurd. From \eqref{eq:plimsup:l1} we deduce that for every $T \in (0,1/2)$ there exists a sequence of times $(a_\eps^T)_\eps \subseteq (0,T)$ such that
	\begin{align*}
		\lim_{\eps \to 0} \norm{\iota_\eps m_{a_\eps^T}^\eps - \mu_{a_\eps^T}}_{\KR(\T^d)} = 0
		\, . 
	\end{align*}
	With a diagonal argument, we find a sequence $(a_\eps)_\eps \subseteq (0,1/2)$ such that
	\begin{align*}
		\lim_{\eps \to 0} a_\eps = 0 
		\tand 
		\lim_{\eps \to 0} \norm{\iota_\eps m_{a_\eps}^\eps - \mu_{a_\eps}}_{\KR(\T^d)} = \lim_{\eps \to 0} \norm{\iota_\eps m_{a_\eps}^\eps - \mu_0}_{\KR(\T^d)} = 0
		\, .
	\end{align*}
	Similarly, we can find another sequence $(b_\eps)_\eps \subseteq (1/2,1)$ such that
	\begin{align*}
		\lim_{\eps \to 0} b_\eps = 1 \quad \text{and} \quad \lim_{\eps \to 0} \norm{\iota_\eps m_{b_\eps}^\eps - \mu_1}_{\KR(\T^d)}=0
		\, .
	\end{align*}
	We claim the sought recovery sequences is provided by $m_0^\eps \coloneqq m^\eps_{a_\eps}$ and $m_1^\eps \coloneqq m_{b_\eps}^\eps$. In order to show this, let us define $\hat J^\eps : \cE_{\eps} \to \R$ via the formula\footnote{The definition is well-posed because $\eps R_0$ is assumed to be smaller than $1/2$.} (recall the assumption \eqref{eq:ass_limsup})
	\begin{align*}
		\frac{\tau_\eps^z \hat J^\eps}{\eps^{d-1}} \coloneqq \bar J \, , \qquad z \in \Z^d_\eps
		\, , 
	\end{align*}
	so that $\hat J^\eps$ is divergence-free. Now define
	\begin{align*}
		\tilde m^\eps_t \coloneqq 
		\begin{cases}
			m^\eps_{a_\eps} &\text{if } t \in [0,a_\eps) \\
			m^\eps_t &\text{if } t \in [a_\eps,b_\eps] \\
			m^\eps_{b_\eps} &\text{if } t \in (b_\eps,1]
		\end{cases}
		\tand
		\tilde J^\eps_t \coloneqq 
		\begin{cases}
			\hat J^\eps &\text{if } t \in [0,a_\eps) \\
			J^\eps_t &\text{if } t \in [a_\eps,b_\eps] \\
			\hat J^\eps &\text{if } t \in (b_\eps,1]
		\end{cases}	\, .
	\end{align*}
	It is readily verified that $(\tilde \bfm^\eps, \tilde \bfJ^\eps)$ solves the continuity equation for every $\eps$. Therefore
	\begin{align}
		\label{eq:main_est_limsup}
		\cMA_\eps(m_0^\eps,m_1^\eps) 
		&\leq 
		\cA_\eps(\tilde \bfm^\eps,\tilde \bfJ^\eps) 
		= \int_0^1 \sum_{z \in \Z^d_\eps} \eps^d F\left( \frac{\tau^z_\eps \tilde m_t^\eps}{\eps^d}, \frac{\tau^z_\eps \tilde J^\eps_t}{\eps^{d-1}} \right) \, \dd t 
		\\
		&= \int_0^{a_\eps} \sum_{z \in \Z^d_\eps} \eps^d F\left( \frac{\tau^z_\eps m_{a_\eps}^\eps}{\eps^d}, \bar J \right) \, \dd t + \int_{a_\eps}^{b_\eps} \sum_{z \in \Z^d_\eps} \eps^d F\left( \frac{\tau^z_\eps m_t^\eps}{\eps^d}, \frac{\tau^z_\eps J^\eps_t}{\eps^{d-1}} \right) \, \dd t 
		\\
		&\quad + \int_{b_\eps}^1 \sum_{z \in \Z^d_\eps} \eps^d F\left( \frac{\tau^z_\eps m_{b_\eps}^\eps}{\eps^d}, \bar J \right) \, \dd t 
		\\
		&\eqqcolon I_1 + I_2 + I_3
		\, .
	\end{align}
	The first and last integral can be estimated using the assumption \eqref{eq:ass_limsup}. Indeed,
	\begin{align*}
		I_1+I_3 
		&\le C \sum_{z \in \Z^d_\eps} \Biggl((a_\eps+1-b_\eps) \eps^d  + \sum_{\substack{x \in \mathcal X \\ |x_\sz|_{\infty} \le R}} \left( a_\eps (\tau^z_\eps m_{a_\eps}^\eps)(x) + (1-b_\eps) (\tau^z_\eps m_{b_\eps}^\eps)(x) \right)\Biggr) 
		\\
		&\le C \left( (a_\eps+1-b_\eps) + (2R+1)^d \sum_{x \in \X_\eps} \left(a_\eps m_{a_\eps}^\eps(x) + (1-b_\eps) m_{b_\eps}^\eps(x)\right) \right) 
		\\
		&= C \left( (a_\eps+1-b_\eps) + (2R+1)^d \left(a_\eps \iota_\eps m_{a_\eps}^\eps(\T^d) + (1-b_\eps)\iota_\eps m_{b_\eps}^\eps(\T^d)\right) \right)
		\, , 
	\end{align*}
	and in the limit we find
	\begin{align}
		\label{eq:mit_est_limsup} 
		\limsup_{\eps \to 0} I_1+I_3 \le C\left(0 + (2R+1)^d(0\cdot \mu_0(\T^d) + 0 \cdot \mu_1(\T^d))\right) = 0 
		\, .
	\end{align}
	As for the second integral, thanks to Assumption~\ref{ass:F}(c) we have that
	\begin{align}
		\label{eq:est_limsup} 
		I_2 - \mathcal A_\eps(\bfm^\eps,\bfJ^\eps) 
		&= -\int \int_{(0,a_\eps)\cup(b_\eps,1)} \sum_{z \in \Z^d_\eps} \eps^d F\left( \frac{\tau^z_\eps m_t^\eps}{\eps^d}, \frac{\tau^z_\eps J^\eps_t}{\eps^{d-1}} \right) \, \dd t 
		\\
		&\le C' \left((a_\eps+1-b_\eps)+(2R+1)^d \iota_\eps\bfm^\eps\left(\left((0,a_\eps)\cup(b_\eps,1)\right) \times \T^d\right) \right) 
		\, .
	\end{align}
	Since $(\iota_\eps \bfm^\eps)_\eps$ converges weakly, for every $a , b \in (0,1)$, we have that 
	\begin{align*}
		\limsup_{\eps \to 0}
		\iota_\eps\bfm^\eps
		\left(
		\left(
		(0,a_\eps)\cup(b_\eps,1)
		\right) 
		\times \T^d
		\right) 
		&\leq  
		\limsup_{\eps \to 0}
		\iota_\eps \bfm^\eps 
		\left(
		\big(
		(0,a) \cup (b,1)
		\big)
		\times \Td  
		\right)
		\\
		&\leq 
		\bfmu
		\left(
		\big(
		(0,a) \cup (b,1)
		\big)
		\times \Td  
		\right)
		\, .
	\end{align*}
	Using the fact that the previous estimate holds for every $a,b \in (0,1)$,
	we obtain that
	\begin{align*}
		\limsup_{\eps \to 0}
		\iota_\eps\bfm^\eps
		\left(
		\left(
		(0,a_\eps)\cup(b_\eps,1)
		\right) 
		\times \T^d
		\right) 
		= 0 
		\, .
	\end{align*}
	This, together with the estimate obtained in \eqref{eq:est_limsup}, gives us the inequality  
	\begin{align}
		\label{eq:final_est_limsup}
		\limsup_{\eps \to 0} I_2 \le \limsup_{\eps \to 0} \cA_\eps(\bfm^\eps, \bfJ^\eps)
		\, .
	\end{align}
	In conclusion, from \eqref{eq:main_est_limsup}, \eqref{eq:mit_est_limsup}, and \eqref{eq:final_est_limsup} we find
	\begin{align*}
		\limsup_{\eps \to 0} 
		\cMA_\eps(m_0^\eps,m_1^\eps) 
		\le 
		\limsup_{\eps \to 0} 
		\cA_\eps(\bfm^\eps, \bfJ^\eps) \le \bA(\bfmu) 
		\le 
		\bMA(\mu_0,\mu_1) + 1/k
		\,  ,
	\end{align*}
	which is sought upper bound.
\end{proof}

\subsection{Proof of the liminf inequality}
In this section, we provide the proof of the liminf inequality in Theorem~\ref{theo:main}.   Let $m_0^\eps$, $m_1^\eps$ be a sequence of measures weakly converging to $\mu_0$, $\mu_1$, respectively. We want to show that
\begin{align}
	\label{eq:claim_liminf}
	\liminf_{\eps \to 0}
	\cMA_\eps(m_0^\eps,m_1^\eps)
	\geq 
	\bMA_\hom(\mu_0,\mu_1)
	\, .
\end{align}
Note that we may assume that~$m_0^\eps(\cX_\eps)=m_1^\eps(\cX_\eps)$ for every~$\eps>0$.

We split the proof into two parts: first for $F$ with linear growth and then for $F$ of flow-based type, respectively Assumption~\ref{ass:F_lineargr} and Assumption~\ref{ass:F_fbased}.
\subsubsection{Case 1: $F$ with linear growth}
Assume that $F$ satisfies Assumption~\ref{ass:F_lineargr}. Recall that, as a consequence of Lemma~\ref{lemma:lip_disc}, $F$ is Lipschitz continuous as well, in the sense of \eqref{eq:Lip_F_disc}.
\begin{proof}[Proof of the liminf inequality (linear growth)]
	With a very similar argument as the one provided by Remark~\ref{rem:nonnegative} in the continuous setting, we can with no loss of generality assume that $F$ is nonnegative. Moreover, up to extracting a subsequence, we may assume that the limit inferior in~\eqref{eq:claim_liminf} is a true finite limit.
	Let $(\bfm^\eps,\bfJ^\eps) \in \cCE_\eps$ be approximate optimal solutions associated to $\cMA_\eps(m_0^\eps,m_1^\eps)$, i.e. such that
	\begin{align}	\label{eq:optimal_curve}
		\lim_{\eps \to 0} \bigl( \cA_\eps(\bfm^\eps,\bfJ^\eps) - \cMA_\eps(m_0^\eps,m_1^\eps) \bigr) = 0 \, .
	\end{align}
	As usual, we write $\dd \bfm^\eps (t,x) = m_t^\eps(\dd x) \dd t$ for some measurable curve $t \mapsto m_t^\eps \in \R_+^{\cX_\eps}$ of constant, finite mass.
	By compactness (Remark~\ref{rem:compactness}), we know that up to a further non-relabeled subsequence, $\bfm^\eps \to \bfmu$ weakly in $\cMT$ with $\bfmu \in \BVKR$, as well as $\bfJ^\eps \to \bfnu$ weakly in $\cM^d\bigl((0,1) \times \Td\bigr)$, for some $(\bfmu, \bfnu) \in \CE$. Due to the lack of continuity of the trace operators in $\BV$, a priori we cannot conclude that $\bfmu_{t=0} = \mu_0$ and $\bfmu_{t=1} = \mu_1$. In other words, there might be a ``jump'' in the limit as $\eps \to 0$ at the boundary of $(0,1)$. In order to take care of this problem, we rescale our measures $\bfm^\eps$ in time, so as to be able to ``see'' the jump in the interior of $(0,1)$, where we are able to estimate the energy.
	
	To this purpose, for $\delta \in (0,1/2)$,  we define $\cI_\delta:= (\delta,1-\delta)$ and $\bfm^{\eps,\delta} \in \BVKR$ as 
	\begin{align}
		\label{eq:rescale_disc}
		m_t^{\eps,\delta} := 
		\begin{cases}
			m_0^\eps	
			&\text{if } t \in (0,\delta] \\
			m_{\frac{t-\delta}{1-2\delta}}^\eps 
			&\text{if } t \in \cI_\delta \\
			m_1^\eps
			&\text{if } t \in [1-\delta,1)
		\end{cases} 
		\, , \quad 
		\dd \bfm^{\eps,\delta}(t,x) := m_t^{\eps,\delta}(\dd x) \dd t 
		\, .
	\end{align}
	By construction, the convergence of the boundary data, and the fact that, by assumption, $\bfm^\eps \to \bfmu$ weakly, it is straightforward to see that $\bfm^{\eps,\delta} \to \hat \bfmu^\delta$ weakly, where 
	\begin{align}
		\label{eq:def_muhat}
		\hat \mu_t^\delta := 
		\begin{cases}
			\mu_0
			&\text{if } t \in (0,\delta] \\
			\mu_{\frac{t-\delta}{1-2\delta}}
			&\text{if } t \in \cI_\delta \\
			\mu_1
			&\text{if } t \in [1-\delta,1)
		\end{cases} 
		\, , \quad 
		\dd \hat \bfmu^\delta(t,x) := \hat \mu_t^\delta(\dd x) \dd t 
		\, .
	\end{align}
	Note that the rescaled curve $t \mapsto \hat \mu_t^\delta$ might have discontinuities at $t=\delta$ and $t= 1-\delta$, which correspond to the possible jumps in the limit as $\eps \to 0$ for $\bfm^\eps$ at $\{0,1\}$. Nevertheless, $\hat \bfmu^\delta$ is a competitor for $\bMA(\mu_0,\mu_1)$, which, by the $\Gamma$-convergence of $\cA_\eps$ towards $\bA_\hom$ (Remark~\ref{rem:gammaconv}), ensures that 
	\begin{align}	\label{eq:first_est_lsc}
		\liminf_{\eps \to 0}
		\cA_\eps(\bfm^{\eps,\delta}) \geq \bA_\hom(\hat \bfmu^\delta) \geq \bMA_\hom(\mu_0,\mu_1)
		\, .
	\end{align}
	We are left with estimating from above the left-hand side of the latter displayed equation. To do so, we seek a suitable curve of discrete vector fields $\bfJ^{\eps,\delta}$ so that $(\bfm^{\eps,\delta}, \bfJ^{\eps,\delta}) \in \CE$ and whose energy $\cA_\eps(\bfm^{\eps,\delta},\bfJ^{\eps,\delta})$ is comparable with $\cA_\eps(\bfm^{\eps},\bfJ^{\eps})$ for small $\delta>0$. It is useful to introduce the following notation: for $\delta \in (0,1/2)$, 
	\begin{align}
		r_\delta: (0,1) \to \cI_\delta
		\, , \quad 
		r_\delta(s) = (1-2\delta)s + \delta \, ,
		\\	
		R_\delta:(0,1) \times \Td \to \cI_\delta \times \Td 
		\, , \quad 
		R_\delta(s,x) = (r_\delta(s), x) 
		\, . 
	\end{align}
	
	We note that
	\[
	{\iota_\eps \bfm^{\eps,\delta}}|_{\cI_\delta \times \Td} = (1-2\delta) (R_\delta)_\# \iota_\eps \bfm^\eps.
	\]
	Indeed, for any test function~$\varphi \in C_b$, we have
	\begin{align*}
		\int_{\cI_\delta \times \Td} \varphi \dd \iota_\eps \bfm^{\eps, \delta}
		&=
		\int_{\delta}^{1-\delta} \int_{\Td} \varphi(t,x) \dd \iota_\eps m_t^{\eps,\delta} \dd t
		=
		(1-2\delta) \int_0^1 \int_{\Td} \varphi(r_\delta(s),x) \dd \iota_\eps m_s^{\eps} \dd s \\
		&=
		(1-2\delta) \int_{(0,1) \times \Td} \varphi \circ R_\delta \dd \iota_\eps \bfm^\eps
		=
		(1-2\delta) \int_{(0,1) \times \Td} \varphi \dd (R_\delta)_\# \iota_\eps \bfm^\eps
	\end{align*}
	Furthermore, we have the following lemma.
	
	\begin{lemma}	\label{lemma:change_var}
		Let $\bfsigma \in \cM_+\bigl( (0,1) \times \T^d \bigr)$ be a singular measure with respect to $\Leb^{d+1}$. Then, the measure $(R_\delta)_{\#} \bfsigma \in \cM_+ (\cI_\delta \times \Td)$ is also singular with respect to $\Leb^{d+1}$. Moreover, for every measure $\bfxi = f \Leb^{d+1} + f^\perp \bfsigma \in \cM^n((0,1) \times \Td)$, we have the decomposition 
		\begin{align}
			(R_\delta)_{\#} \bfxi = f^\delta \Leb^{d+1} + f^{\delta,\perp}  (R_\delta)_{\#} \bfsigma 
			\, , 
		\end{align}
		where the respective densities are given by the formulas
		\begin{align}
			f^\delta(t,x) 
			=  \frac1{1-2\delta} f\bigl(r_\delta^{-1}(t), x\bigr)
			\tand 
			f^{\delta,\perp}(t,x) 
			=  f^\perp\bigl(r_\delta^{-1}(t), x\bigr)
			\, .
		\end{align}
	\end{lemma}
	\begin{proof}
		By assumption, $\bfsigma$ is singular with respect to $\Leb^{d+1}$, which means there exists a set $A \subset (0,1) \times \Td$ such that $\Leb^{d+1}(A) = 0 =  \bfsigma(A^c)$. By the very definition of push-forward and the bijectivity of $R_\delta$, we then have that 
		\begin{align}
			(R_\delta)_{\#} \bfsigma \big((R_\delta(A))^c \big) 
			=
			\bfsigma \Big( R_\delta^{-1} \big(R_\delta(A^c) \big)  \Big) 
			= 
			\bfsigma(A^c) 
			= 0 \, , 
		\end{align}
		whereas, by the scaling properties of the Lebesgue measure, we have that $\Leb^{d+1}(R_\delta(A)) = (1-2\delta) \Leb^{d+1}(A) = 0$, which shows the claimed singularity. The second part of the lemma follows from the fact that $(R_\delta)_{\#} \Leb^{d+1} = (1-2\delta)^{-1} \Leb^{d+1}$ and the following statement: for every $\bfxi' = f' \bfsigma'$ with $\bfsigma' \in  \cM_+((0,1) \times \Td)$, we claim that 
		\begin{align}	\label{eq:claim_density}
			\frac
			{\dd (R_\delta)_{\#} \bfxi'}
			{\dd (R_\delta)_{\#} \bfsigma'}
			(t,x) = f'(R_\delta^{-1}(t,x))
			\, , \quad \forall (t,x) \in \cI_\delta \times \Td \, .
		\end{align}
		Indeed, by definition of push-forward, we have for every test function $\varphi \in C_b$ 
		\begin{align*}
			\int \varphi \dd (R_\delta)_{\#} \bfxi'
			= 
			\int (\varphi \circ R_\delta) \dd \bfxi'
			= 
			\int (\varphi \circ R_\delta) f' \dd \bfsigma'
			= 
			\int \varphi \cdot (f' \circ R_\delta^{-1}) \dd (R_\delta)_{\#} \bfsigma'
			\, ,
		\end{align*}
		which indeed shows \eqref{eq:claim_density}. 
	\end{proof}
	We continue the proof of the liminf inequality by defining the sought $\bfJ^{\eps,\delta} \in \cM^d\bigl((0,1)\times \Td\bigr)$ as follows: if $\hat \iota_\delta: \cM^d(\cI_\delta\times \Td) \to \cM^d\bigl((0,1)\times \Td\bigr)$ is the natural embedding obtained by extending to $0$ any measure  outside $\cI_\delta$, then the sought curve of vector fields is uniquely determined by 
	\begin{align}
		\label{eq:def_hat_iota}
		\iota_\eps \bfJ^{\eps,\delta} := \hat \iota_\delta  \big[ (R_\delta)_{\#} (\iota_\eps \bfJ^\eps) \big] \in  \cM^d\bigl((0,1)\times \Td\bigr)
		\, ,
	\end{align}
	that is,
	\begin{align*}
		J_t^{\eps,\delta} := 
		\begin{cases}
			0	
			&\text{if } t \in (0,\delta] \\
			\displaystyle \frac{1}{1-2\delta} J_{\frac{t-\delta}{1-2\delta}}^\eps 
			&\text{if } t \in \cI_\delta \\
			0
			&\text{if } t \in [1-\delta,1)
		\end{cases} \, .
	\end{align*}
	
	We claim that $	( \bfm^{\eps,\delta}, \bfJ^{\eps,\delta}  ) \in \cCE_\eps$ and  
	\begin{align}	\label{eq:claim_final}
		\bA(\bfm^{\eps,\delta}, \bfJ^{\eps,\delta}) 
		\leq
		\bA(\bfm^\eps, \bfJ^\eps) 
		+ C(F) \delta 
		\Big(
		(\iota_\eps \bfm^\eps) ((0,1) \times \Td)  + |\iota_\eps \bfJ^\eps|((0,1) \times \Td) + 1
		\Big)
		\, ,  
	\end{align}
	where $C(F) \in \R_+$ only depends on $F$ (specifically on $C$ as given in the linear growth assumption~\ref{ass:F_lineargr}). This would suffice to conclude the proof of the sought liminf inequality. Indeed, from \eqref{eq:optimal_curve} and \eqref{eq:claim_final} we infer
	\begin{align}
		\nonumber
		\liminf_{\eps \to 0} 
		&\cMA_\eps(m_0^\eps , m_1^\eps) 
		=
		\liminf_{\eps \to 0} 
		\cA_\eps(\bfm^\eps, \bfJ^\eps) 
		\\ 
		\nonumber
		&\geq 
		\liminf_{\eps \to 0}	
		\cA_\eps(\bfm^{\eps,\delta}, \bfJ^{\eps,\delta}) 
		- C(F) \delta 
		\Big(
		( \iota_\eps \bfm^\eps ) ((0,1) \times \Td)  + |\iota_\eps  \bfJ^\eps|((0,1) \times \Td) + 1
		\Big) 
		\nonumber
	\end{align}
	which, combined with \eqref{eq:first_est_lsc}, yields
	\[
		\liminf_{\eps \to 0} 
		\cMA_\eps(m_0^\eps , m_1^\eps) 
			\geq 
		\bMA_\hom(\mu_0,\mu_1) 
		- C(F) \delta 
		\Big(
		\mu_0(\Td)  + |\bfnu|((0,1) \times \Td) + 1
		\Big)
	\]
	for any $\delta \in (0,1/2)$. We conclude by letting~$\delta \to 0$.
	
	We are left with the proof of~$	( \bfm^{\eps,\delta}, \bfJ^{\eps,\delta}  ) \in \cCE_\eps$ and of the claim \eqref{eq:claim_final}.
	
	\smallskip 
	\noindent
	\textit{Proof of} $	( \bfm^{\eps,\delta}, \bfJ^{\eps,\delta}  ) \in \cCE_\eps$. \ 
	Let us fix~$x \in \X_\eps$ and~$\varphi \in C^1_c \bigl((0,1)\bigr)$. Set $\tilde \varphi : = \varphi \circ r_\delta$. We have
	\begin{align} \label{eq:CE_proof}
		\nonumber
		\int_0^1 \partial_t \varphi m_t^{\eps,\delta}(x) \dd t
		&=
		\int_0^\delta \partial_t \varphi \, m_0^{\eps}(x) \dd t
		+
		\int_{1-\delta}^1 \partial_t \varphi \, m_1^\eps(x) \dd t
		+
		\int_{\cI_\delta} \partial_t \varphi \, m_{r_\delta^{-1}(t)}^\eps(x) \dd t \\
		\nonumber
		&=
		\varphi(\delta) \, m_0^\eps(x) - \varphi(1-\delta) \, m_1^\eps(x) + (1-2\delta)\int_0^1 (\partial_t \varphi) \circ r_\delta \, m_s^\eps(x) \dd s \\
		\nonumber
		&=
		\tilde \varphi(0) \, m_0^\eps(x) - \tilde \varphi(1) \, m_1^\eps(x) + \int_0^1 \partial_s \tilde \varphi \, m_s^\eps(x) \dd s \\
		\nonumber
		&=
		\int_0^1 \tilde \varphi \, \sum_{y \sim x} J_s^\eps(x,y) \dd s
		=
		\frac{1}{1-2\delta} \int_{\cI_\delta} \varphi \, \sum_{y \sim x} J^\eps_{r_\delta^{-1}(t)}(x,y) \dd t \\
		&=
		\int_0^1 \varphi \sum_{y \sim x} J^{\eps,\delta}_t(x,y) \dd t \, ,
	\end{align}
	where, in the fourth equality, we used that $(\bfm^\eps, \bfJ^\eps) \in \cCE_\eps$.
	
	\smallskip 
	\noindent
	\textit{Proof of the energy estimate}. \ 
	Note that, by construction, for $(t,(x,y)) \in \cI_\delta \times \cE_\eps$, 
	\begin{align}
		\label{eq:form_rescaled}
		m_t^{\eps,\delta}(x) =  m_{r_\delta^{-1}(t)}^\eps(x) 
		\, , \qquad 
		J_t^{\eps,\delta}(x,y) = \frac1{1-2\delta}   J_{ r_\delta^{-1}(t)}^\eps(x,y) 
		\, . 
	\end{align}
	On the other hand, for $(t,(x,y)) \in \big( (0,\delta] \cup [1-\delta,1) \big) \times \cE_\eps$, we have that 
	\begin{align*}
		m_t^{\eps,\delta}(x) =
		\begin{cases}
			m_0^\eps(x)
			&\text{if } t \in (0,\delta]  \\
			m_1^\eps(x)
			&\text{if } t \in [1-\delta,1) 
		\end{cases} 
		\tand 
		J_t^{\eps,\delta}(x,y) = 0
		\, .
	\end{align*}
	It follows that the energy of $(\bfm^{\eps,\delta}, \bfJ^{\eps,\delta}) $ is given by 
	\begin{align}
		\label{eq:energy_rescaled}	
		\cA_\eps(\bfm^{\eps,\delta}, \bfJ^{\eps,\delta})  
		= 
		\int_0^1 
		\cF_\eps(m_t^{\eps,\delta}, J_t^{\eps,\delta})
		\dd t
		= 	\cA_\eps^{\cI_\delta} 	(\bfm^{\eps,\delta}, \bfJ^{\eps,\delta})
		+
		\delta \sum_{i=0,1}
		\cF_\eps(m_i^\eps,0)  
		\, , 
	\end{align}
	where we used the notation 
	\begin{align}	\label{eq:energy_rescaled2}
		\cA_\eps^{\cI_\delta}(\bfm^{\eps,\delta}, \bfJ^{\eps,\delta}) 
		&:=
		\int_{\cI_\delta} 
		\cF_\eps(m_t^{\eps,\delta}, J_t^{\eps,\delta})
		\dd t
		=
		(1-2\delta)
		\int_0^1
		\cF_\eps\Big( m_t^\eps, \frac1{1-2\delta}J_t^\eps \Big)
		\dd t
		\, .
	\end{align} 
	Using Assumption~\ref{ass:F_lineargr}, we see that, for $i=0,1$, 
	\begin{align*}
		\cF_\eps(m_i^\eps,0) 
		\leq 
		C (m_i^\eps(\cX_\eps) +1) 
		=
		C \bigl( \iota_\eps \bfm^\eps \bigl((0,1)\times \Td\bigl) + 1\bigr) 
	\end{align*}
	and, by the Lipschitz continuity exhibited in Lemma~\ref{lemma:lip_disc}, we also infer that
	\begin{align*}
		\cA_\eps^{\cI_\delta}(\bfm^{\eps,\delta}, \bfJ^{\eps,\delta}) 
		&\leq 
		(1-2\delta) \left(\cA_\eps(\bfm^\eps, \bfJ^\eps)
		+
		C \Big(\frac1{1-2\delta}-1 \Big)
		\int_0^1  \| \iota_\eps J_t^\eps \|_{\TV} \dd t \right)
		\\
		&=
		(1-2\delta) \cA_\eps(\bfm^\eps, \bfJ^\eps)
		+
		2 \delta C |\iota_\eps \bfJ^\eps|\bigl((0,1)\times \Td\bigr) \, .
	\end{align*}
	Since we assumed~$F$ to be nonnegative, we can further estimate
	\[
	(1-2\delta) \cA_\eps(\bfm^\eps, \bfJ^\eps)
	\le
	\cA_\eps(\bfm^\eps, \bfJ^\eps)
	\]
	and, combining these estimates with~\eqref{eq:energy_rescaled}, we find
	\begin{align*}
		\cA_\eps(\bfm^{\eps,\delta}, \bfJ^{\eps,\delta})  
		\leq 
		\cA_\eps(\bfm^\eps, \bfJ^\eps)
		+ 2 \delta 
		\Big(
		\iota_\eps \bfm^\eps \bigl((0,1)\times \Td\bigr) + C |\iota_\eps \bfJ^\eps|\bigl((0,1)\times \Td\bigr)  + 1
		\Big)
		\, , 
	\end{align*}
	which concludes the proof of \eqref{eq:claim_final}.
\end{proof}

\subsubsection{Case 2: F is flow-based}
In this section we show \eqref{eq:claim_liminf} in the case $F$ (and hence $f_\hom$) is of flow-based type, i.e.~it satisfies Assumption~\ref{ass:F_fbased}. We start by observing that, in this special setting, both discrete and continuous formulation of the boundary-value problems admit an equivalent, static formulation. 

Let $(\bfmu,\bfnu) \in \CE$, and consider the Lebesgue decomposition 
\begin{align*}
	\bfmu = \rho \Leb^{d+1} + \rho^\perp \bfsigma \, , \qquad
	\bfnu = j \Leb^{d+1} + j^\perp \bfsigma \, .
\end{align*}
We know that every solution to the continuity equation can be disintegrated in the form $\bfmu(\dd t,\dd x) = \mu_t(\dd x) \dd t$ for some measurable curve $t \mapsto \mu_t \in \cM_+(\Td)$ of constant, finite mass.
If $f$ is a function as in Assumption~\ref{ass:f} that further  does \emph{not} depend on~$\rho$, then Jensen's inequality yields  
\begin{align}	\label{eq:JensenA}
	\int_0^1 \int_{\T^d} f(j_t) \dd x \dd t \ge \int_{\T^d} f\left(\int_0^1 j_t \dd t\right) \dd x \, .
\end{align}
In order to take care of the singular part, consider the disintegration of $\bfsigma$ with respect to the projection map $\pi:(t,x) \mapsto x$, in the form
\begin{align}	\label{eq:def_pi}
	\bfsigma(\dd t, \dd x) = \sigma^x(\dd t, \dd x') (\pi_{\#}\bfsigma)(\dd x) \, , 
\end{align}
for some measurable $x \mapsto \sigma^x \in \Prob\big( (0,1) \times \Td \big)$ so that $\sigma^x$ is concentrated on $(0,1) \times \{x\}$ for $\pi_{\#} \bfsigma$-a.e.~$x \in \T^d$. Due to the convexity of $f^\infty$, by Jensen's inequality we also obtain
\begin{align}	\label{eq:JensenA2}
	\int_{(0,1) \times \Td} f^\infty(j^\perp) \dd \bfsigma \ge \int_{\T^d} f^\infty\left(\int j^\perp \dd \sigma^x \right) \dd \pi_{\#} \bfsigma (x) \, .	
\end{align}
Now, we define the new space-time measures 
\begin{align}	\label{eq:affineA}
	\begin{split}
		\tilde \bfmu := \tilde \mu_t(\dd x) \dd t 
		\quad \text{and} \quad 
		\tilde \bfnu := \hat j \Leb^{d+1} + \hat j^\perp \dd t \otimes \pi_{\#}\bfsigma \, , 
		\quad \text{where} 
		\\ 
		\tilde \mu_t:= \mu_0 + t (\mu_1 - \mu_0)
		\, , \quad 
		\hat j(x) := \int_0^1 j_t(x) \dd t   
		\, , \quad \text{and} \ \  
		\hat j^\perp(x) := \int j^\perp \dd \sigma^x
		\, .
	\end{split}
\end{align}
By \eqref{eq:JensenA} and \eqref{eq:JensenA2}, we therefore have
\begin{align} \label{eq:flowbased1}
	\bA(\bfmu, \bfnu) \geq 
	\int_{\T^d} f(\hat j) \dd x
		+
	\int_{\T^d} f^\infty(\hat j^\perp) \dd \pi_{\#}\bfsigma(x)
		\, .
\end{align}
We need to be careful here: the decomposition of~$\tilde \bfnu$ in~\eqref{eq:affineA} may not a the Lebesgue decomposition, in the sense that~$\dd t \otimes \pi_\# \bfsigma$ can have a nonzero absolutely continuous part. Let~$\tilde \sigma \in \cM_+(\Td)$ be singular w.r.t.~$\Leb^d$ and such that~$\mu_0,\mu_1, \pi_\# \bfsigma \ll \Leb^d + \tilde \sigma$. We can write Lebesgue decompositions
\[
\tilde \bfmu
=
\tilde \rho \Leb^{d+1} + \tilde \rho^\perp \dd t \otimes \tilde \sigma \, ,
\qquad
\tilde \bfnu = \tilde j \Leb^{d+1} + \tilde j^\perp \dd t \otimes \tilde \sigma \, .
\]
If we write
\[ \pi_\# \bfsigma = \alpha \Leb^d + \beta \tilde \sigma \]
for some functions~$\alpha, \beta : \Td \to \R_+$, then
\[
\tilde j
=
\hat j + \alpha \hat j^\perp
\quad \text{and} \quad
\tilde j^\perp
=
\beta \hat j^\perp \, .
\]
The inequality \eqref{eq:flowbased1} becomes, recalling that $f^\infty$ is $1$-homogeneous,
\begin{equation} \label{eq:flowbased2}
	\bA(\bfmu, \bfnu)
	\geq
	\int_{\Td} \bigl( f(\hat j) + f^\infty(\alpha \hat j^\perp) \bigr) \dd x + \int_{\Td} f^\infty (\beta \hat j^\perp) \dd \tilde \sigma \, .
\end{equation}
At this point, we need a lemma.
\begin{lemma}
	For every~$j_1,j_2 \in \R^d$, we have that 
	$
		f(j_1+j_2) \le f(j_1) + f^\infty(j_2)
	$.
\end{lemma}

\begin{proof}
	Let~$g \le f$ be a convex and Lipschitz continuous function. By convexity, for every~$\epsilon \in (0,1)$, we have  
	\[
	g(j_1+j_2)
	=
	g\left((1-\epsilon)\frac{j_1}{1-\epsilon} + \epsilon \frac{j_2}{\epsilon}\right)
	\le
	(1-\epsilon) g\left( \frac{j_1}{1-\epsilon} \right) + \epsilon g\left(\frac{j_2}{\epsilon} \right) \, .
	\]
	Let~$j_0 \in \Dom(f)$. By the Lipschitz continuity of $g$,
	\[
	g(j_1+j_2)
	\le
	(1-\epsilon) \left( g(j_1) + (\Lip g) \left( \frac{1}{1-\epsilon} - 1 \right) \abs{j_1} \right) + \epsilon g\left( \frac{j_2}{\epsilon} + j_0 \right) + \epsilon (\Lip g) \abs{j_0}
	\]
	and, since~$g \le f$,
	\[
	g(j_1 + j_2)
	\le
	(1-\epsilon) f(j_1)  + \epsilon f \left( \frac{j_2}{\epsilon} + j_0 \right)+ \epsilon (\Lip g) \bigl( \abs{j_0}+\abs{j_1} \bigr) \, .
	\]
	As we let~$\epsilon \to 0$, we find
	\[
	g(j_1 + j_2)
	\le
	f(j_1) + f^\infty(j_2) \, .
	\]
	Since~$f$ is convex and lower semicontinuous, we conclude by an approximation argument.
\end{proof}
Applying this lemma with~$j_1 = \hat j(x)$ and~$j_2 = \alpha \hat j^\perp(x)$ for every~$x \in \Td$,~\eqref{eq:flowbased2} finally becomes
\[
\bA(\bfmu, \bfnu)
\ge
\int_{\Td} f(\tilde j) \dd x + \int_{\Td} f^\infty(\tilde j^\perp) \dd \tilde \sigma
=
\bA(\tilde \bfmu, \tilde \bfnu) \, .
\]
In other words, we have shown that an optimal curve $\bfmu$ between two given boundary data is always given by the affine interpolation (and a constant-in-time flux $\hat j$). We conclude that
\begin{gather}
	\label{eq:static_contA}
	\bMA(\mu_0,\mu_1) = \bA(\boldsymbol {\tilde \mu})
	\\
	\nonumber
	= \inf_\nu
	\left\{
	\int_{\Td} f(\tilde j) \dd x 
	+ 
	\int_{\Td} f^\infty(\tilde j^\perp) \dd \tilde \sigma 
	\suchthat 
	\nu = \tilde j \Leb^d + \tilde j^\perp \tilde \sigma \, , \ \Leb^d \perp \tilde \sigma 
	\ \ \text{and} \ \   
	\nabla \cdot \nu = \mu_0 - \mu_1 
	\right\} 
	\, . 
\end{gather}
We refer to the latter expression as the \textit{static formulation} of the boundary value problem described by $\bMA(\mu_0,\mu_1)$ (in the case when $f$ is of flow-based type).
\begin{remark}
	\label{rem:flow-based} 
	Using such equivalence, the lower semicontinuity of $\bMA$ directly follows from Assumption~\ref{ass:f} (which provides weak compactness for the fluxes), the fact that the constraint in \eqref{eq:static_contA} is closed in $\mu_0, \mu_1, \nu$ w.r.t.~the weak topology, and the semicontinuity of the map
	\[
	\nu \mapsto \int_{\Td} f(\tilde j) \dd x 
	+ 
	\int_{\Td} f^\infty(\tilde j^\perp) \dd \tilde \sigma 
	\]
	ensured by \cite[Theorem 2.34]{Ambrosio-Fusco-Pallara:2000} (see also \cite[Lemma 3.14]{Gladbach-Kopfer-Maas-Portinale:2023}).
\end{remark} 
Arguing in a similar way (in fact, via an even simpler argument, due to the lack of singularities), we obtain a static formulation of the discrete transport problem in terms of a discrete divergence equation, when $F(m,J)=F(J)$. Precisely, in this case we obtain 
\begin{align}
	\label{eq:static_disc}
	\cMA_\eps(m_0,m_1)  
	= \inf 
	\left\{
	\cF_\eps(J)  
	\suchthat 
	J \in \R_a^{\cE_\eps} 
	\, , \quad    
	\dive J = m_0 - m_1 
	\right\} 
	\, . 
\end{align}
The sought $\Gamma$-liminf inequality easily follows from such static formulations, in particular \eqref{eq:static_contA}.

\begin{proof}[Proof of the liminf inequality (flow-based type)]
	Let $m_0^\eps$, $m_1^\eps \in \R_+^{\cX_\eps}$ be a sequence of discrete nonnegative measures which converge weakly (via $\iota_\eps$ in the usual sense) to $\mu_0$, $\mu_1$, and such that~$m_0^\eps(\cX_\eps) = m_1^\eps(\cX_\eps)$ for every~$\eps>0$. Let $(\bfm^\eps, \bfJ^\eps) \in \cCE_\eps$ the (almost-)optimal solutions associated with $\cMA_\eps(m_0^\eps,m_1^\eps)$, namely 
	\begin{align}
		\label{eq:estimate_static} 
		\liminf_{\eps \to 0}
		\cMA_\eps(m_0^\eps, m_1^\eps)
		= 
		\liminf_{\eps \to 0}
		\cA_\eps(\bfm^\eps)
		=
		\liminf_{\eps \to 0}
		\cA_\eps(\bfm^\eps, \bfJ^\eps)
		\, .
	\end{align}
	Consider the discrete equivalent of the measure constructed in \eqref{eq:affineA},  namely
	\begin{align*}
		\tilde m_t^\eps := m_0^\eps + t (m_1^\eps - m_0^\eps) 
		\tand 
		\tilde J^\eps_t \equiv  \tilde J^\eps:= \int_0^1 J_s^\eps \dd s
		\, , 
	\end{align*}
	which still solves the continuity equation.
	By applying Jensen's inequality, the convexity of $F$ ensures that $(\tilde \bfm^\eps , \tilde \bfJ^\eps )$ has a lower energy, i.e. 
	\begin{align}
		\label{eq:estimate_static_2}
		\cA_\eps(\bfm^\eps , \bfJ^\eps)
		\geq 
		\cA_\eps(\tilde \bfm^\eps , \tilde \bfJ^\eps)
		=
		\cF_\eps(\tilde J^\eps)
		\tand 
		(\tilde \bfm^\eps, \tilde \bfJ^\eps) \in \cCE_\eps
		\, .
	\end{align}
	Without loss of generality, we assume that $\sup_\eps \cF_\eps(\tilde J^\eps) < \infty$, which, thanks to Assumption~\ref{ass:F} (see e.g.~\cite[Equation~6.2]{Gladbach-Kopfer-Maas-Portinale:2023} for a similar argument),
	implies that
	\begin{align*}
		\sup_{\eps > 0} \| \iota_\eps \tilde J^\eps \|_{\TV(\Td)} < \infty
		\, .
	\end{align*}
	Therefore, up to a (non-relabeled) subsequence, we have $\tilde J^\eps \to \tilde \nu \in \cM^d\big(\T^d\big)$ weakly as $\eps \to 0$. In particular, we have that, by construction, 
	\begin{align*}
		(\tilde \bfm^\eps, \tilde \bfJ^\eps) \to (\tilde \bfmu, \tilde \bfnu) \in \CE 
	\end{align*}
	weakly, where the limit measures are given by 
	\begin{align*}
		\tilde \bfmu := \tilde \mu_t(\dd x) \dd t 
		\quad \text{with} \quad 
		\tilde \mu_t:= \mu_0 + t (\mu_1 - \mu_0)
		\, ,  \tand  
		\tilde \bfnu := \tilde \nu(\dd x) \dd t 
		\, .
	\end{align*}
	Further, we have that $\nabla \cdot \tilde \nu = \mu_0 - \mu_1$. All in all, from \eqref{eq:estimate_static}, \eqref{eq:estimate_static_2}, and the $\Gamma$-convergence of $\cA_\eps$ to $\bA_\hom$ (Remark~\ref{rem:gammaconv}), we infer that 
	\begin{align*}
		\liminf_{\eps \to 0}
		\cMA_\eps(m_0^\eps, m_1^\eps)
		\geq 
		\bA_\hom(\tilde\bfmu, \tilde \bfnu)	
		\geq 
		\bMA_\hom( \mu_0 , \mu_1 )
		\, , 
	\end{align*}
	which concludes the proof of the liminf inequality.
\end{proof}

\subsection{About the lower semicontinuity of $\bMA$}
\label{sec:lsc}
In view of our main result, whenever $F$ satisfies either Assumption~\ref{ass:F_lineargr} or Assumption~\ref{ass:F_fbased}, the limit boundary-value problem $\bMA_\hom(\cdot, \cdot)$ is necessarily jointly lower semicontinuous with respect to the weak topology on $\cM_+(\Td) \times \cM_+(\Td)$. This indeed follows from the general fact that any $\Gamma$-limit with respect to a given topology is always lower semicontinuous with respect to that same topology. Using a very similar proof to that of the $\Gamma$-liminf inequality, we can actually show that, if $f$ is with linear growth or it is of flow-based type, then the associated $\bMA$ is always lower semicontinuous (even if, a priori, $f$ is not of the form $f=f_\hom$), thus providing a positive answer in this framework to the validity of \eqref{eq:liminf_boundary}.
In the flow-based setting, this fact has been  observed in Remark~\ref{rem:flow-based}.

Assume now that $f$ is with linear growth, namely it satisfies one of the two equivalent conditions appearing in Lemma~\ref{lemma:lip}. We consider $(\mu_0^n$, $\mu_1^n) \to (\mu_0$, $\mu_1) \in \cM_+(\Td) \times \cM_+(\Td)$ weakly and claim that 
\begin{align*}
	\liminf_{n \to \infty}
	\bMA(\mu_0^n, \mu_1^n) \geq \bMA(\mu_0,\mu_1)
	\, .
\end{align*}
The proof goes along the same line of the proof of the $\Gamma$-liminf inequality for discrete energies $F$ with linear growth. We sketch it here and add details whenever we encounter nontrivial differences between the two proofs. 

\begin{proof}
	Let $(\bfmu_n,\bfnu_n) \in \CE$ be (almost-)optimal solutions associated to $\bMA(\mu_0^n,\mu_1^n)$, i.e.,
	\begin{align}	\label{eq:optimal_curve_lsc}
		\liminf_{n \to \infty}
		\bMA(\mu_0^n,\mu_1^n) 
		=
		\liminf_{n \to \infty}
		\bA(\bfmu^n,\bfnu^n) \, .
	\end{align}
	With no loss of generality we can assume $\sup_n \bA(\bfmu^n, \bfnu^n) <\infty$ and that the limits inferior are true limits. Hence, by compactness (Remark~\ref{rem:compactness}), we know that up to a non-relabeled subsequence, $\bfmu^n \to \bfmu$ weakly in $\cM_+\bigl( (0,1) \times \Td \bigr)$, as well as $\bfnu^n \to \bfnu$ weakly in $\cM^d\bigl((0,1) \times \Td \bigr)$, with $(\bfmu, \bfnu) \in \CE$. 
	Moreover, we also have $\dd \bfmu(t,x) = \mu_t(\dd x) \dd t \in \BVKR$ for some measurable curve $t \mapsto \mu_t \in \cM_+(\Td)$ of constant, finite mass.
	Once again, due to the lack of continuity of the trace operators in $\BV$, we cannot ensure that $\bfmu_{t=0} = \mu_0$ and $\bfmu_{t=1} = \mu_1$. To solve this issue, we rescale our measures $\bfmu^n$ in time in the same spirit as in \eqref{eq:rescale_disc}. For a given $\delta>0$,  we define $\cI_\delta:= (\delta,1-\delta)$ and $\bfmu^{n,\delta} \in \BVKR$ as 
	\begin{align}
		\mu_t^{n,\delta} := 
		\begin{cases}
			\mu_0^n	
			&\text{if } t \in (0,\delta] \\
			\mu_{\frac{t-\delta}{1-2\delta}}^n
			&\text{if } t \in \cI_\delta \\
			\mu_1^n
			&\text{if } t \in [1-\delta,1)
		\end{cases} 
		\, , \quad 
		\dd \bfmu^{n,\delta}(t,x) := \mu_t^{n,\delta}(\dd x) \dd t 
		\, .
	\end{align}
	By construction, it is not hard to see that $\bfmu^{n,\delta} \to \hat \bfmu^\delta$ weakly, where 
	\begin{align}
		\hat \mu_t^\delta := 
		\begin{cases}
			\mu_0
			&\text{if } t \in (0,\delta] \\
			\mu_{\frac{t-\delta}{1-2\delta}}
			&\text{if } t \in \cI_\delta \\
			\mu_1
			&\text{if } t \in [1-\delta,1)
		\end{cases} 
		\, , \quad 
		\dd \hat \bfmu^\delta(t,x) := \hat \mu_t^\delta(\dd x) \dd t 
		\, .
	\end{align}
	We stress that, as in \eqref{eq:def_muhat}, the rescaled curve $t \mapsto \hat \mu_t^\delta$ could have discontinuities at $t=\delta$ and $t= 1-\delta$, corresponding to the possible jumps in the limit as $n \to \infty$ for $\bfmu^n$ at $\{0,1\}$. Nevertheless, $\hat \bfmu^\delta$ is a competitor for $\bMA(\mu_0,\mu_1)$, which, by lower semicontinuity of $\bA$, ensures that 
	\begin{align}	\label{eq:first_est_lsc_lsc}
		\liminf_{n \to \infty}
		\bA(\bfmu^{n,\delta}) \geq \bA(\hat \bfmu^\delta) \geq \bMA(\mu_0,\mu_1)
		\, .
	\end{align}
	In order to estimate the left-hand side of the latter displayed equation, we seek a suitable vector meausure $\bfnu^{n,\delta}$ so that $(\bfmu^{n,\delta}, \bfnu^{n,\delta}) \in \CE$ and whose energy $\bA(\bfmu^{n,\delta},\bfnu^{n,\delta})$ is comparable with $\bA(\bfmu^n,\bfnu^n)$ for small $\delta>0$. 
	
	Recall the definitions of $r_\delta$, $R_\delta$, their properties from Lemma~\ref{lemma:change_var}, and the definition of $\hat \iota_\delta$ as in \eqref{eq:def_hat_iota}.
	We set
	\begin{align}
		\nu^{n,\delta} := \hat \iota_\delta  \big[ (R_\delta)_{\#} \bfnu^n \big] \in  \cM^d((0,1)\times \Td)
		\, . 
	\end{align}
	The proof that $( \bfmu^{n,\delta}, \bfnu^{n,\delta}  ) \in \CE$ works exactly as in  \eqref{eq:CE_proof}. In the same spirit as in \eqref{eq:claim_final}, we claim that
	\begin{align}	\label{eq:claim_final_lsc}
		\bA(\bfmu^{n,\delta}, \bfnu^{n,\delta}) 
		\leq
		\bA(\bfmu^n, \bfnu^n) 
		+ C(f) \delta 
		\Big(
		\bfmu^n ((0,1) \times \Td)  + |\bfnu^n|((0,1) \times \Td) + 1
		\Big)
		\, ,  
	\end{align}
	where $C(f) \in \R_+$ only depends on $f$ (specifically on $\Lip f$). The combination of \eqref{eq:optimal_curve_lsc}, \eqref{eq:first_est_lsc_lsc} and \eqref{eq:claim_final_lsc}, and the arbitrariness of~$\delta$ would then suffice to conclude the proof.
	
	We are left with the proof of the claim \eqref{eq:claim_final_lsc}, which is a bit more involved, compared to that of~\eqref{eq:claim_final}, due to the presence of the singular part at the continuous level.
	We apply Lemma~\ref{lemma:change_var} to both $\bfmu^n$ and $\bfnu^n$ and find that
	\begin{gather*}
		\bfmu^{n,\delta} = \rho^{n,\delta} \dd \Leb^{d+1} + \rho^{n, \delta, \perp} \dd (R_\delta)_{\#} \bfsigma
		\tand 
		\bfnu^{n,\delta} = j^{n,\delta} \dd \Leb^{d+1} + j^{n, \delta, \perp} \dd (R_\delta)_{\#} \bfsigma  
		\, , 
	\end{gather*}
	where $(R_\delta)_{\#} \bfsigma$ is singular with respect to $\Leb^{d+1}$ and, for $(t,x) \in \cI_\delta \times \Td$,  
	\begin{align}
	\label{eq:densitiesPF}
		\rho^{n,\delta}(t,x) &=  \big( \rho^n \circ R_\delta^{-1}  \big) (t,x) 
		\, , \qquad 
		&&\rho^{n, \delta, \perp}(t,x) = (1-2\delta) \big( \rho^{n,\perp} \circ R_\delta^{-1} \big) (t,x)
		\, , \nonumber
		\\
		j^{n,\delta}(t,x) &= \frac1{1-2\delta}  \big( j^n \circ R_\delta^{-1}  \big) (t,x) 
		\, , \qquad 
		&&j^{n, \delta, \perp}(t,x) =  \big( j^{n,\perp} \circ R_\delta^{-1}  \big)
		\, . 
	\end{align}
	On the other hand, for $(t,x) \in \big( (0,\delta] \cup [1-\delta,1) \big) \times \Td$, we have that 
	\begin{align*}
		\rho^{n,\delta}(t,x) =
		\begin{cases}
			\rho_0^n(x)
			&\text{if } t \in (0,\delta) \\
			\rho_1^n(x)
			&\text{if } t \in [1-\delta,1) 
		\end{cases} \, ,
		\qquad 
		\rho^{n, \delta, \perp}(t,x)= 
		\begin{cases}
			\rho_0^{n,\perp}(x)
			&\text{if } t \in (0,\delta]  \\
			\rho_1^{n,\perp}(x)
			&\text{if } t \in [1-\delta,1) 
		\end{cases} \, ,
	\end{align*}
	and
	$
	j^{n,\delta}(t,x) = j^{n, \delta, \perp}(t,x) = 0
	$
	. 
	It follows that the energy of $(\bfmu^{n,\delta}, \bfnu^{n,\delta}) $ is given by\footnote{Note that the definition of the energy does not depend on the choice of the measure which is singular with respect to $\Leb^{d+1}$, therefore we can use $(R_\delta)_{\#} \bfsigma$ instead of $\bfsigma$.}
	\begin{align}\nonumber	
		\bA(\bfmu^{n,\delta}, \bfnu^{n,\delta})  
		&= 
		\int_{(0,1) \times \Td} 
		f(\rho^{n,\delta}, j^{n,\delta})
		\dd \Leb^{d+1}
		+\int_{ (0,1) \times \Td} 
		f^\infty( \rho^{n,\delta,\perp}, j^{n,\delta,\perp} )
		\dd (R_\delta)_{\#} \bfsigma
		\\ \label{eq:energy_rescaled_lsc}
		&=
		\bA^{\cI_\delta} 	(\bfmu^{n,\delta}, \bfnu^{n,\delta})
		+ 
		\sum_{i=0,1}
		\delta
		\Big( 
		\int_{\Td} f(\rho_i^n,0) \dd \Leb^d  
		+
		\int_{\Td} f^\infty(\rho_i^{n,\perp},0) \dd \pi_{\#} \bfsigma   
		\Big) 
		\, , \qquad 
	\end{align}
	where $\pi$ is defined as in \eqref{eq:def_pi} and we used the notation 
	\[
		\bA^{\cI_\delta}(\bfmu^{n,\delta}, \bfnu^{n,\delta}) 
		:=
		\int_{\cI_\delta \times \Td} 
		f(\rho^{n,\delta}, j^{n,\delta})
		\dd \Leb^{d+1}
		+
		\int_{ \cI_\delta \times \Td} 
		f^\infty( \rho^{n,\delta,\perp}, j^{n,\delta,\perp} )
		\dd (R_\delta)_{\#} \bfsigma \, .
	\]
	Making use of the formulas~\eqref{eq:densitiesPF} and the homogeneity of $f^\infty$, we find
	\begin{align}	\label{eq:energy_rescaled2_lsc}
		\bA^{\cI_\delta}&(\bfmu^{n,\delta}, \bfnu^{n,\delta})  
		=
		(1-2\delta)
		\int_{(0,1) \times \Td} 
		f\Big( \rho^n, \frac{j^n}{1-2\delta} \Big)
		\dd \Leb^{d+1} 
		+ \int_{(0,1) \times \Td} 
		f^\infty\big( (1-2\delta) \rho^{n,\perp}, j^{n,\perp} \big) \dd \bfsigma \nonumber
		\\ 
		&\qquad =
		(1-2\delta)
		\bigg( 
		\int_{(0,1) \times \Td} 
		f\Big( \rho^n, \frac{j^n}{1-2\delta} \Big)
		\dd \Leb^{d+1} 
		+ \int_{(0,1) \times \Td} 
		f^\infty\Big( \rho^{n,\perp}, \frac{j^{n,\perp}}{1-2\delta} \Big) \dd \bfsigma 
		\bigg)
		\, .
	\end{align}
	Furthermore, it follows from the linear growth assumption that, for $i=0,1$, 
	\begin{align*}
		\int_{\Td} f(\rho_n^i,0) \dd \Leb^d 
		+ 
		\int_{\Td} f^\infty(\rho_i^{n,\perp},0) \dd \pi_{\#} \bfsigma 
		\leq 
		C (\mu_n^i(\Td) +1) 
		=
		C ( \bfmu^n ((0,1) \times \Td) + 1) 
	\end{align*}
	as well as,  from \eqref{eq:energy_rescaled2_lsc}  and by the nonnegativity of $f$,
	\begin{align*}
		\bA^{\cI_\delta}(\bfmu^{n,\delta}, \bfnu^{n,\delta}) 
		&\leq 
		\bA(\bfmu^n, \bfnu^n)
		+
		\Big( \frac1{1-2\delta} - 1 \Big) (\Lip f )
		\bigg(
		\int_{(0,1) \times \Td} \| j^n \| \dd \Leb^{d+1}
		+ 
		\int_{(0,1) \times \Td} \| j^{n,\perp} \| \dd \bfsigma 	 
		\bigg)
		\\
		&\leq
		\bA(\bfmu^n, \bfnu^n)
		+
		2 \delta (\Lip f) |\bfnu^n|((0,1) \times \Td) \, .
	\end{align*}
	Recalling \eqref{eq:energy_rescaled_lsc}, we conclude that 
	\begin{align*}
		\bA(\bfmu^{n,\delta}, \bfnu^{n,\delta})  
		\leq 
		\bA(\bfmu^n, \bfnu^n)
		+ 2 \delta 
		\Big(
		\bfmu^n ((0,1) \times \Td) + (\Lip f) |\bfnu^n|((0,1) \times \Td)  + 1
		\Big)
		\, , 
	\end{align*}
	and thus \eqref{eq:claim_final_lsc}.
\end{proof}

\section{Analysis of the cell problem with examples}
\label{sec:analsys_cellprob}
This section is devoted to the characterisation and illustration of~$f_\hom$ in the case where the function~$F$ is of the form
\begin{align}
\label{eq:F_example}
F(m,J)
=
F(J)
=
\sum_{ (x, y) \in \EQ } \alpha_{xy} \abs{J(x,y)}
\end{align}
for some strictly positive function~$\alpha : \EQ \ni (x,y) \mapsto \alpha_{xy}>0$. A natural problem of interest is to determine whether/when the~$\Gamma$-limit~$\bMA_\hom$ can be the~$W_1$-distance. The analogous problem for the~$W_2$-distance has been extensively studied in~\cite{GlKoMa18} and~\cite{Gladbach-Kopfer-Maas-Portinale:2023} in the case where the graph stucture is associated with finite-volume partitions.

\subsection{Discrete $1$-Wasserstein distance}

We start the analysis of this special setting by observing that, in this case, the discrete functional $\cMA_\eps$ actually coincides with the $\bW_1$ distance associated to a natural induced metric structure. In order to prove this fact, we first define~$\tilde \alpha^\eps : \cE_\eps \to \R_+$ as the unique function such that
\[
	\frac{\tau_\eps^z \tilde \alpha^\eps}{\eps} \Big |_{\EQ} \coloneqq \alpha \, \qquad z \in \Z_\eps^d \, .
\]
It is easy to check that this definition is well-posed and determines the value of~$\tilde \alpha_{xy}$ for every~$(x,y) \in \cE_\eps$. Further let
\[
	\alpha_{xy}^\eps
		=
	\frac{\tilde \alpha_{xy}^\eps + \tilde \alpha_{yx}^\eps}{2} \, ,  \qquad (x,y) \in \cE_\eps
\]
be the symmetrisation of~$\tilde \alpha^\eps$. Given~$J \in \R^{\cE_\eps}_a$, we can write~$\cF_\eps(J)$ in terms of~$\alpha^\eps$. Precisely,
\begin{align*}
	\cF_\eps(J)
		&=
	\sum_{z \in \Z^d_\eps} \eps^d F \left( \frac{\tau_\eps^z J }{\eps^{d-1}} \right)
		=
	\sum_{z \in \Z^d_\eps} \eps^d \sum_{(\hat x, \hat y) \in \EQ } \alpha_{\hat x \hat y} \frac{ \abs{\tau^z_\eps J(\hat x, \hat y)}} {\eps^{d-1}} \\
		&=
	\sum_{z \in \Z^d_\eps} \sum_{(\hat x, \hat y) \in \EQ} \tau^z_\eps \tilde \alpha^\eps(\hat x, \hat y) \abs{\tau_\eps^z J(\hat x, \hat y)} 
		=
	\sum_{(x,y) \in \cE_\eps} \tilde \alpha^\eps_{xy} \abs{J(x, y)} \\
		&=
	\sum_{(x,y) \in \cE_\eps} \alpha^\eps_{xy} \abs{J(x, y)} \, ,
\end{align*}
where in the last passage we used that~$\abs{J}$ is symmetric.
We define a distance on $\cX_\eps$ given by 
\begin{align}
	d_\eps(x,y):= \cMA_\eps(\delta_x, \delta_y)
		\, , \qquad 
	\forall x,y \in \cX_\eps \, .
\end{align}
One can easily show that $d_\eps$ indeed defines a metric on $\cX_\eps$. In fact,~$d_\eps$ can be seen as a weighted graph distance, in the sense that
\[
	d_\eps(x,y)
		=
	\inf \left\{	\sum_{i=0}^{k-1} 2 \alpha_{x_i x_{i+1}}^\eps
			\suchthat
		x_0 = x \, , \, x_k = y \, , \, (x_i,x_{i+1}) \in \cE_\eps \ \, \, \forall i \, , \, k \in \N
	\right\} \, .
\]

\begin{proof}
	The inequality~$\le$ directly follows by choosing unit fluxes along admissible paths: let~$x_0=x,x_1,\dots,x_{k-1},x_k=y$ be a path, i.e.,~$(x_i,x_{i+1}) \in \cE_\eps$ for every~$i = 0,1,\dots,k$, and consider
	\begin{equation} \label{eq:JP}
		J^P \coloneqq \sum_{i=0}^{k-1} \bigl(\delta_{(x_i,x_{i+1})}-\delta_{(x_{i+1},x_i)} \bigr) \, ,
	\end{equation}
	which has divergence equal to~$\delta_x-\delta_y$.
	Then,
	\begin{align*}
		d_\eps(x,y)
			&=
		\cMA_\eps(\delta_x, \delta_y)
			=
		\inf\left\{ \cF_\eps(J) \suchthat \dive J = \delta_x-\delta_y \right\} \\
			&=
		\inf\left\{ \sum_{(x,y) \in \cE_\eps} \alpha^\eps_{xy} \abs{J(x, y)} \suchthat \dive J = \delta_x-\delta_y \right\} \\
			&\le
		\sum_{(x,y) \in \cE_\eps} \alpha^\eps_{xy} \abs{J^P(x,y)}
			\le
		\sum_{i=0}^{k-1} 2 \alpha_{x_i x_{i+1}}^\eps \, ,
	\end{align*}
	where in the last inequality we used that~$\alpha^\eps$ is symmetric. 
	
	To prove the converse, let~$\bar J \in \R^{\cE_\eps}_a$ be an optimal flux for~$\cMA_\eps(\delta_x,\delta_y)$, that is,
	\begin{equation*}
		\dive \bar J = \delta_x-\delta_y \tand 
		\cMA_\eps(\delta_x,\delta_y) = \sum_{ (x, y) \in \cE_\eps } \alpha_{xy}^\eps \abs{\bar J(x,y)}
		\, .
	\end{equation*}
	Since the graph~$\cE_\eps$ is finite, in order for~$\bar J$ to satisfy the divergence condition, there must exist a simple path~$x_0 = x, x_1,\dots,x_k = y$ such that~$(x_i,x_{i+1}) \in \cE_\eps$ and~$\bar J(x_i,x_{i+1}) > 0$ for every~$i$. Let~$J^P$ be the associated vector field as in~\eqref{eq:JP}. Note that, for every~$\lambda \in \R$, we have~$\dive \bigl((1-\lambda) \bar J + \lambda J^P\bigr) = \delta_x - \delta_y$. Furthermore, the function
	\[
		\lambda
			\mapsto
		\sum_{(x,y) \in \cE_\eps} \alpha_{xy}^\eps \abs{(1-\lambda)\bar J(x,y) + \lambda J^P(x,y)}
	\]
	is differentiable at~$\lambda = 0$, since~$(\bar J(x,y) = 0) \Rightarrow (J^P(x,y) = 0)$. By optimality, the derivative at~$\lambda = 0$ must equal~$0$, i.e.,
	\[
		0= \sum_{ (x, y) \in \cE_\eps } \alpha_{xy}^\eps \bigl( J^P(x,y) - \bar J(x,y)	\bigr) \sgn \bar J(x,y) = \sum_{ (x, y) \in \cE_\eps } \alpha_{xy}^\eps J^P(x,y)  \sgn \bar J(x,y) - d_\eps(x,y) \, ,
	\]
	and, since~$(J^P(x,y) \neq 0) \Rightarrow (\sgn J^P(x,y) = \sgn \bar J(x,y))$, we have
	\[
		d_\eps(x,y) = \sum_{(x,y) \in \cE_\eps} \alpha_{xy}^\eps \abs{J^P(x,y)} = 2\sum_{i=1}^k \alpha_{x_i x_{i+1}}^\eps \, ,
	\]
	where, in the last equality, we used that the path is simple (and the symmetry of~$\alpha^\eps$). This shows the $\geq$ inequality and concludes the proof. 
\end{proof}

Consider the $1$-Wasserstein distance associted to~$d_\eps$, that is,
\begin{align}
	\bW_{1,\eps}(m_0,m_1) = 
	\inf
	\left\{
		\int_{\cX_\eps \times \cX_\eps} d_\eps(x,y) \dd \pi(x,y) 
			\suchthat 
		(e_0)_{\#} \pi = m_0
			\, , \quad 
		(e_1)_{\#} \pi = m_1 
	\right\}
		\, , 
\end{align}
as well as, by Kantorovich duality,
\begin{align}
	\bW_{1,\eps}(m_0,m_1) =
	\sup
	\left\{
		\int_{\cX_\eps} \varphi \dd (m_0-m_1)
			\suchthat 
		\Lip_{d_\eps}(\varphi) \leq 1
	\right\}
		\, , 
\end{align} 
for every $m_0,m_1 \in \Prob(\cX_\eps)$. We claim that, in fact, 
\begin{align}
\label{eq:MA-W1}
	\cMA_\eps(m_0,m_1) = \bW_{1,\eps}(m_0,m_1)
		\, , \qquad  
	\forall m_0,m_1 \in \Prob(\cX_\eps)
		\, .
\end{align}

\begin{proof}[Proof of~$\ge$] Fix $m_0,m_1 \in \Prob(\cX_\eps)$ and set $m:= m_0-m_1$. Let $\bar J \in \R_a^{\cE_\eps}$ be an optimal flux for $\cMA_\eps(m_0,m_1)$, that it,
\begin{align}
	\dive \bar J = m \tand 
	\cMA_\eps(m_0,m_1) = \sum_{ (x, y) \in \cE_\eps } \alpha_{xy}^\eps \abs{\bar J(x,y)}
		\, .
\end{align}
Let $\varphi: \cX_\eps \to \R$ be such that $\Lip_{d_\eps} \varphi \leq 1$, i.e.,~$
	\left|
		\varphi(y) - \varphi(x)
	\right|
		\leq 
	d_\eps(x,y)
$ for 
$
x, y \in \cX_\eps
		\, . 
$
Then, 
\begin{align}
	\int_{\cX_\eps} \varphi \dd m 
		&= 
	\int_{\cX_\eps} \varphi \dd \dive \bar J 
		= 
	\sum_{x \in \cX_\eps} \varphi(x) \sum_{y \sim x} \bar J(x,y) 
		= 
	\frac12 
	\sum_{(x,y) \in \cE_\eps} \varphi(x) 
		\big( \bar J(x,y) - \bar J(y,x) \big)
\\
		&=
	\frac12 
	\sum_{(x,y) \in \cE_\eps} 
		\big( \varphi(y) - \varphi(x) \big) \bar J(x,y)
		=
	\sum_{(x,y) \in \cE_\eps \suchthat \bar J(x,y)>0}
		\big( \varphi(y) - \varphi(x) \big) \bar J(x,y)	 
\\
		&\leq 
	\sum_{(x,y) \in \cE_\eps \suchthat \bar J(x,y)>0}
		d_\eps(x,y) \bar J(x,y)
			\, .
\end{align}
In order to conclude, we make the following crucial observation: as a consequence of the optimality of $\bar J$, we claim that
\begin{equation} \label{eq:implic}
	\bar J(x,y) > 0  
		\quad \Longrightarrow \quad
	d_\eps(x,y) = 2 \alpha_{xy}^\eps
		\, . 
\end{equation}
To this end, assume that $\bar J(x,y)>0$ and consider an optimal $J^{(x,y)}$ for $d_\eps(x,y) = \cMA_\eps(\delta_x, \delta_y)$. Note that, by construction,
\begin{align}
	\dive
	\big(
	J^{(x,y)}
	\big) 
		=
	\delta_x - \delta_y 
		=
	\dive 
		\tilde J 
	\, , \qquad 
	\text{where} \ \tilde J := \delta_{(x,y)} - \delta_{(y,x)}
		\, ,
\end{align}
which in turns also implies that
\begin{align}
	\dive
	\big(
		\bar J + \bar J(x,y) 
		\big(
			J^{(x,y)} - \tilde J
		\big)
	\big)
		= 
	\dive \bar J
		\, .
\end{align}
By optimality of $J^{(x,y)}$, we have
\begin{equation} \label{eq:tildeJopt}
	F(\tilde J) = 2 \alpha_{xy}^\eps 
		\geq
	F \big( J^{(x,y)} \big)
		= 
	\sum_{(\tilde x , \tilde y) \in \cE_\eps}
		\alpha_{\tilde x, \tilde y}^\eps \abs{J^{(x,y)}(\tilde x, \tilde y)}
			\, , 
\end{equation}
whereas the optimality of $\bar J$ yields
\begin{align*}
	F\bigl(\bar J + \bar J(x,y) 
	\big(
	J^{(x,y)} - \tilde J
	\big)\bigr)
		&=
	\sum_{(\tilde x, \tilde y) \in \cE_\eps \setminus \{(x,y), (y,x)\} } \alpha_{\tilde x \tilde y}^\eps \abs{ \bar J(\tilde x, \tilde y) +
	\bar J(x,y) J^{(x,y)}(\tilde x, \tilde y)} \\
		&\quad+
	\alpha_{xy}^\eps \abs{\bar J(x,y) J^{(x,y)}(x,y)} + \alpha_{yx}^\eps  \abs{\bar J(y,x) J^{(x,y)}(y,x)} \\
		&\ge
	F(\bar J)
		=
	\sum_{(\tilde x, \tilde y) \in \cE_\eps} \alpha_{\tilde x\tilde y}^\eps \abs{\bar J(\tilde x, \tilde y)} \, .
\end{align*}
	By applying the triangle inequality and simplifying the latter formula, we find
	\begin{equation} \label{eq:barJopt}
		\sum_{(\tilde x,\tilde y) \in \cE_\eps} \alpha_{\tilde x\tilde y}^\eps \abs{\bar J(x,y) J^{(x,y)}(\tilde x,\tilde y)}
			\ge
		2\alpha_{xy}^\eps \abs{\bar J(x,y)} \, .
	\end{equation}
	The combination of~\eqref{eq:tildeJopt} and~\eqref{eq:barJopt} implies $d_\eps(x,y) = 2\alpha_{xy}^\eps$. With~\eqref{eq:implic} at hand, we can write
	\[
		\int_{\cX_\eps} \varphi \dd m \le \sum_{(x,y) \in \cE_\eps \suchthat \bar J(x,y)>0}
		2\alpha_{xy}^\eps \bar J(x,y) = \sum_{(x,y) \in \cE_\eps} \alpha_{xy}^\eps \abs{\bar J(x,y)} = \cMA_\eps(m_0,m_1) \, ,
	\]
	and we conclude by arbitrariness of~$\varphi$.
\end{proof}

\begin{proof}[Proof of~$\le$]
	Let~$\pi$ be such that~$(e_i)_{\#} \pi = m_i$ for~$i=0,1$. Further, for every~$(x,y) \in \cE_\eps$, let~$J^{(x,y)} \in \R^{\cE_\eps}_a$ be optimal for~$\cMA_\eps(\delta_x, \delta_y)$. It follows from a direct computation that the divergence of the asymmetric flux
	\[
		J
			\coloneqq
		\sum_{(x,y) \in \cE_\eps} \pi(x,y) J^{(x,y)}
	\]
	is equal to~$m_0 - m_1$. Thus,
	\[
		\cMA_\eps(m_0,m_1)
			\le
		\sum_{(\tilde x, \tilde y) \in \cE_\eps} \alpha_{\tilde x \tilde y}^\eps \abs{J(\tilde x, \tilde y)}
			\le
		\sum_{(x,y) \in \cE_\eps} \pi(x,y) \sum_{(\tilde x,\tilde y) \in \cE_\eps} \alpha_{\tilde x \tilde y}^\eps  \abs{J^{(x,y)}(\tilde x, \tilde y)}
			=
		\int_{\cX_\eps \times \cX_\eps} d_\eps \dd \pi \, ,
	\]
	and we conclude by arbitrariness of~$\pi$.
\end{proof}
In view of the equality~$\cMA_\eps = \bW_{1,\eps}$, it is worth noting that for energies of the form \eqref{eq:F_example} there are (at least) two different possible methods to show discrete-to-continuum limits for~$\cMA_\eps$. One such method is provided by the current work and makes use of the~$\Gamma$-convergence of~$\cA_\eps$ to~$\bA_\hom$ proved in~\cite[Theorem 5.4]{Gladbach-Kopfer-Maas-Portinale:2023}. The convergence of the ``weighted graph distance''~$d_\eps$ follows \emph{a posteriori}. Another approach is to study directly the scaling limits of the distance $d_\eps$ as $\eps \to 0$ and, from that, infer the convergence of the associated $1$-Wasserstein distances, in a similar spirit as in \cite{Buttazzo-DePascale-Fragal:2001}. 

\subsection{General properties of~$f_\hom$}

For~$j \in \R^d$, recall that
\begin{equation} \label{eq:fhomj}
	f_\hom(j)
	\coloneqq
	\inf \left\{ F(J) \suchthat J \in \Rep(j) \right\} ,
\end{equation}
where~$\Rep(j)$ is the set of all~$\Z^d$-periodic functions~$J \in \R^\cE_a$ such that
\[
\Eff(J) \coloneqq \frac{1}{2} \sum_{(x,y) \in \EQ} J(x,y)(y_\sz - x_\sz) = j
\quad \text{and} \quad
\dive J \equiv 0 \, . \]
As noted in~\cite[Lemma 4.7]{Gladbach-Kopfer-Maas-Portinale:2023}, we may as well write~$\min$ in place of~$\inf$ in~\eqref{eq:fhomj}.

Our first observation is that, indeed, the homogenised density is a norm. This has been already proved in \cite[Corollary~5.3]{Gladbach-Kopfer-Maas-Portinale:2023}; for the sake of completeness we provide here a simple proof in our setting.
\begin{proposition} \label{p:norm}
	The function~$f_\hom$ is a norm.
\end{proposition}

\begin{proof}
	Finiteness follows from the nonemptiness of the set of representatives proved in~\cite[Lemma 4.5]{Gladbach-Kopfer-Maas-Portinale:2023}.
	
	To prove positiveness, take any~$j \in \R^d$ and~$J \in \Rep(j)$. For every norm~$\norm{\cdot}$, we have
	\begin{equation} \label{eq:p:norm:1}
		\norm{j}
		=
		\norm{\Eff(J)}
		\le
		\frac{1}{2}\sum_{(x,y) \in \EQ} \abs{J(x,y)}\norm{y_\sz - x_\sz}
		\le
		\frac{F(J)}{2} \max_{(x,y) \in \EQ} \frac{\norm{y_\sz - x_\sz}}{ \alpha_{xy}} \, .
	\end{equation}
	The constant that multiplies~$F(J)$ at the right-hand side is finite because every~$\alpha_{xy}$ is strictly positive and the graph~$(\cX,\cE)$ is locally finite.
	
	Absolute homogeneity and the triangle inequality follow from the absolute homogeneity and subadditivity of~$F$, and the affinity of the constraints. 
\end{proof}

Hence,~$\bMA_\hom$ is \emph{always} (i.e.~for any choice of~$(\alpha_{xy})_{x,y}$ and of the graph~$(\cX,\cE)$) the~$W_1$-distance w.r.t.~some norm. However, the norm~$f_\hom$ can equal the~$2$-norm~$\abs{\cdot}_2$ only in dimension~$d=1$. In fact, the unit ball for~$f_\hom$ has to be a polytope, namely the associated sphere is contained in finitely many hyperplanes. These types of norms are also known as \textit{crystalline norms}. 

\begin{proposition} \label{p:cryst}
	The unit ball associated to the norm $f_\hom$, namely 
	\[
	B
	\coloneqq
	\left\{ j \in \R^d \suchthat f_\hom(j) \le 1 \right\} \, , 
	\]
	is the convex hull of finitely many points. In particular, the associated unit sphere is contained in finitely many hyperplanes, i.e.,~$f_\hom$ is a crystalline norm.
\end{proposition}

\begin{proof}
	Let~$X$ be the vector space of all~$\Z^d$-periodic functions~$J \in \R^\cE_a$ such that~$\dive J \equiv 0$. The sublevel set 
	\[ 
	X_1 
	\coloneqq 
	\left\{ J \in X \suchthat F(J) \le 1 \right\}
	\]
	is clearly compact (due to the strict positivity of $(\alpha_{xy})_{x,y}$) and can be written as finite intersection of half-spaces, namely
	\[ 
	X_1
	=
	\bigcap_{r \in \{-1,1\}^{\EQ} }
	\left\{ J \in X 
	\suchthat
	\sum_{(x,y) \in \EQ} \alpha_{xy} r_{xy} J(x,y) \le 1 \right\}.
	\]
	Thus,~$X_1$ is the convex hull of some finite set of points~$A$, that is,~$ 	X_1	= \conv(A)$. Since~$f_\hom$ is defined as a minimum, we have
	\[
	B
	=
	\left\{ j \in \R^d \suchthat \exists J \in \Rep(j), \, F(J) \le 1 \right\}
	=
	\Eff(X_1) 
	=
	\Eff(\conv(A)) 
	=
	\conv(\Eff(A))
	\, , 
	\]
	where the last equality is due to the linearity of~$\Eff$.
\end{proof}

\subsection{Embedded graphs}

To visualise some examples, we shall now focus on the case where~$(\cX, \cE)$ is embedded, in the sense that~$V$ is a subset of~$[0,1)^d$ and we use the identification~$(z,v) \equiv z+v$ (also see~\cite[Remark 2.2]{Gladbach-Kopfer-Maas-Portinale:2023}). It has been proved in~\cite[Proposition 9.1]{Gladbach-Kopfer-Maas-Portinale:2023} that, for embedded graphs, the identity
\begin{equation} \label{eq:effgeo}
	\Eff(J)
	=
	\frac{1}{2} \sum_{ (x, y) \in \EQ } J(x,y)(y-x)
\end{equation}
holds for every~$\Z^d$-periodic and divergence-free vector field~$J \in \R^\cE_a$. In what follows, we also make the choice
\[ 
\alpha_{xy} \coloneqq \frac{1}{2}\abs{x-y}_2 \, , \qquad (x,y) \in \EQ \, .
\]

\subsubsection{One-dimensional case with nearest-neighbor interaction} Assume~$d=1$, let~$x_1 < x_2 < \cdots < x_k$ be an enumeration of~$V$, and set
\[
\cE
\coloneqq 
\{ (x,y) \in \X \times \X \text{ s.t.~there is no $z \in \X$ strictly between $x$ and $y$} \}.
\]
In other words, denoting~$x_0 = x_k - 1$ and~$x_{k+1} = x_1 + 1$,
\[
\cE
=
\bigcup_{z \in \Z} \bigcup_{i=1}^k \left\{ (x_i,x_{i+1}) \right\} \cup \left\{ (x_i,x_{i-1}) \right\} .
\]
By rewriting~\eqref{eq:p:norm:1} using~\eqref{eq:effgeo}, and by the definition of~$f_\hom$, we find
\[
\abs{j} \le f_\hom(j) \, , \qquad j \in \R^d \, .
\]
On the other hand, given~$j \in \R^d$, choose
\[
J(x,y)
\coloneqq
j \sgn(y-x) \, ,
\qquad
(x,y) \in \cE \, .
\]
This vector field is in~$\Rep(j)$, because
\[
\dive J(x_i)
=
J(x_i,x_{i+1}) + J(x_i,x_{i-1})
=
j - j
=
0 \, ,
\]
and
\begin{align*}
	\Eff(J)
	&=
	\frac{1}{2} \sum_{i=1}^k \bigl( J(x_i,x_{i+1})(x_{i+1}-x_i) + J(x_i,x_{i-1})(x_{i-1}-x_i) \bigr) \\
	&=
	\frac{j}{2} \sum_{i=1}^k \bigl( \abs{x_{i+1} - x_i} + \abs{x_i - x_{i-1}} \bigr) \\
	&=
	\frac{j}{2} \bigl( x_{k+1} - x_1 + x_k - x_0 \bigr) \\
	&=
	j.
\end{align*}
A similar computation shows that~$F(J) = \abs{j}$, from which~$f_\hom(j) = \abs{j}$.

\subsubsection{Cubic partition} Consider the case where~$\X = \Z^d$ and
\[
\cE
\coloneqq
\left\{ (x,y) \in \Z^d \times \Z^d \suchthat \abs{x-y}_\infty = 1 \right\}.
\]
It is a result of~\cite[Section 9.2]{Gladbach-Kopfer-Maas-Portinale:2023} that
\[
f_\hom(j)
=
\abs{j}_1 \, ,
\qquad
j \in \R^d \, .
\]
Notice that, in this case, the~$2$-norm is evaluated only at vectors on the coordinate axes. Therefore, the same result holds when~$\alpha_{xy} = \frac{1}{2}\abs{x-y}_p$, for any~$p$.

\subsubsection{Graphs in~$\R^2$} A few other examples in dimension~$d=2$ are shown in Figure~\ref{fig:tilings}: for each one, we display the graph and the unit ball in the corresponding norm~$f_\hom$.
To construct algorithmically the unit balls, we solve the variational problem~\eqref{eq:fhomj} for every~$j$ on a discretisation of the circle~$\mathbb S^1$. In turn, this is achieved with the help of the Python library CVXPY \cite{diamond2016cvxpy,agrawal2018rewriting}. For visualisation, we make use of the library matplotlib \cite{Hunter:2007}.

\begin{figure}
	\centering
	\subfloat[][]
	{\includegraphics[width=.43\textwidth, trim={13mm 17mm 20mm 24mm},clip]{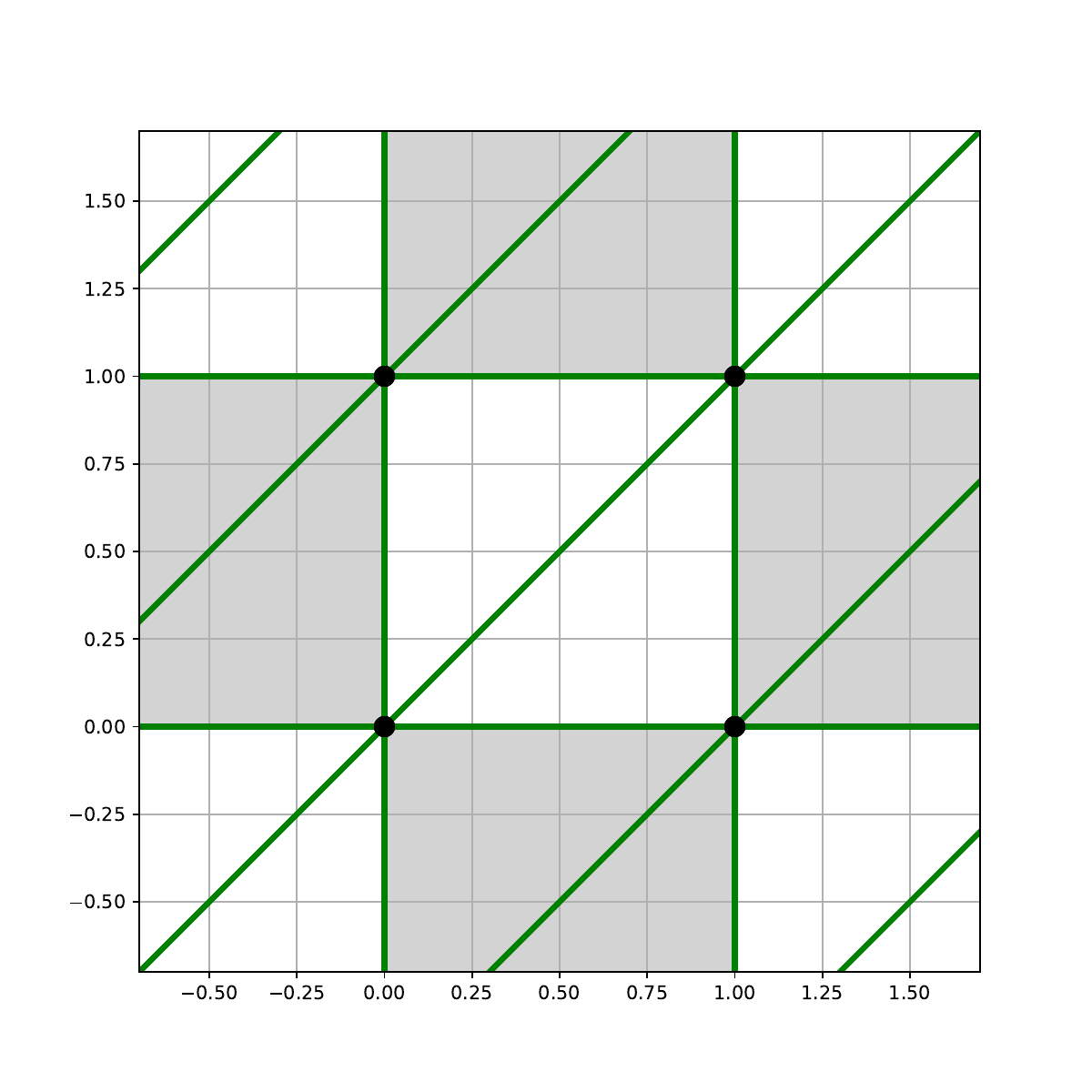}
		\hspace{25pt}
		\includegraphics[width=.43\textwidth, trim={13mm 17mm 20mm 24mm},clip]{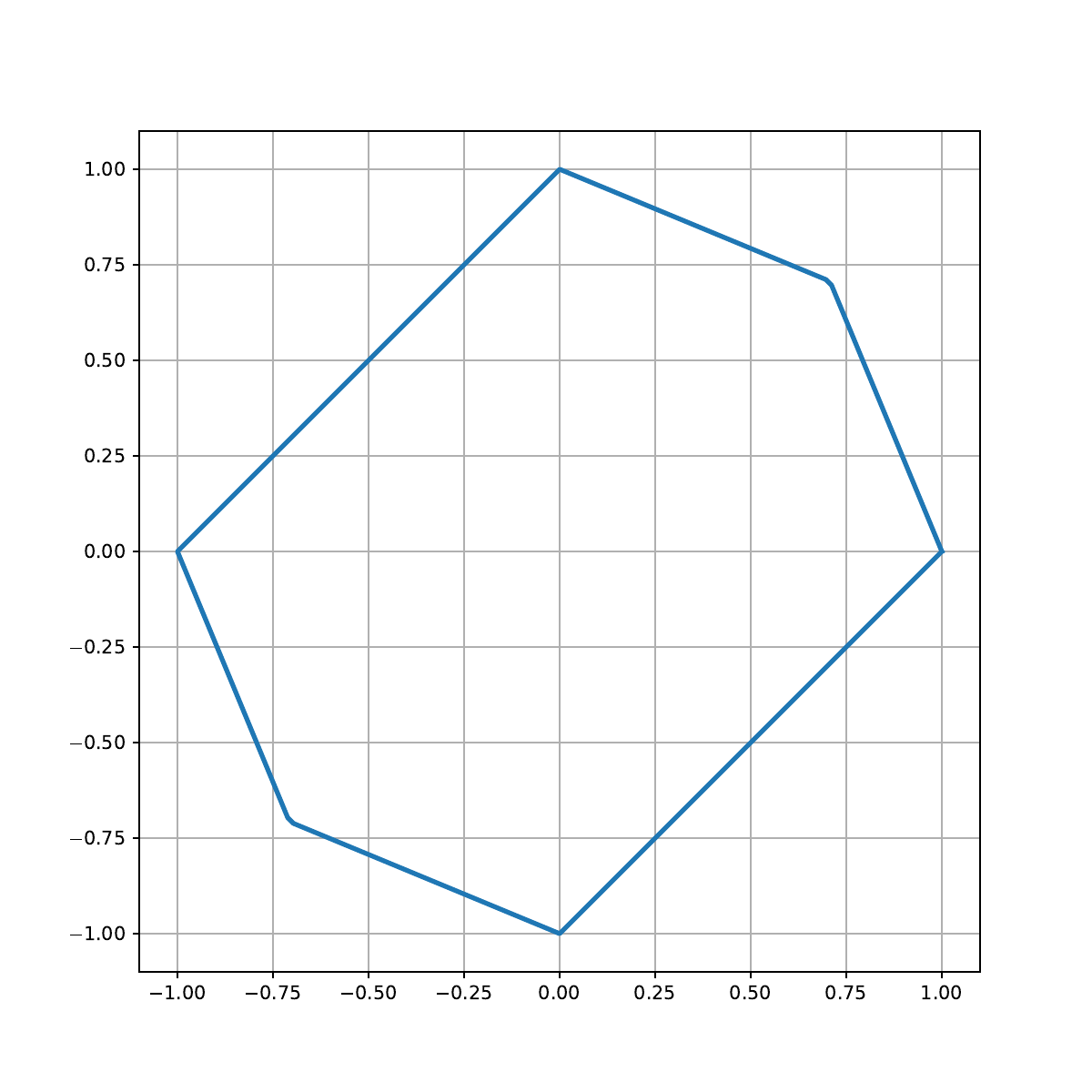}} \\
	\subfloat[][]
	{\includegraphics[width=.43\textwidth, trim={13mm 17mm 20mm 24mm},clip]{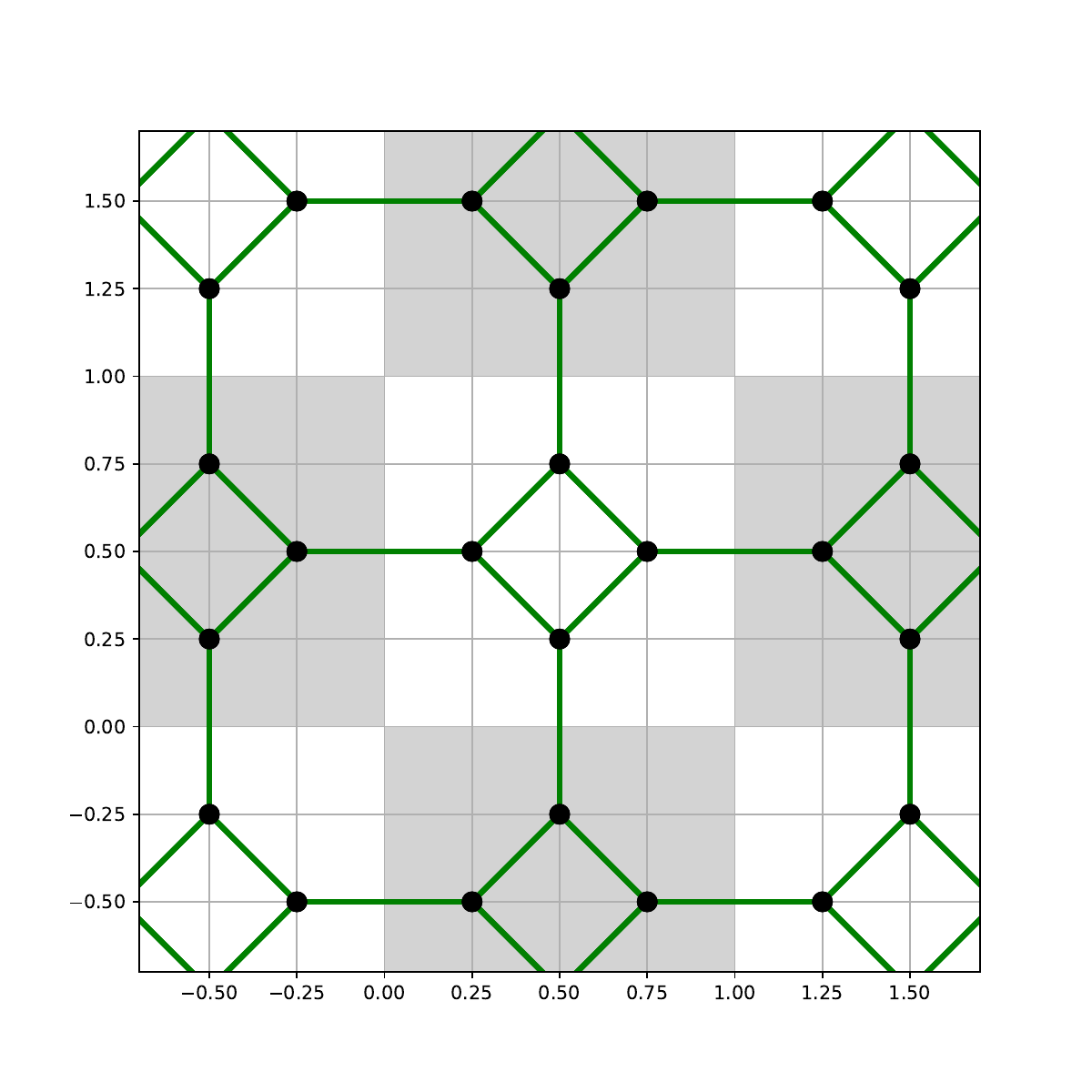}
		\hspace{25pt}
		\includegraphics[width=.43\textwidth, trim={13mm 17mm 20mm 24mm},clip]{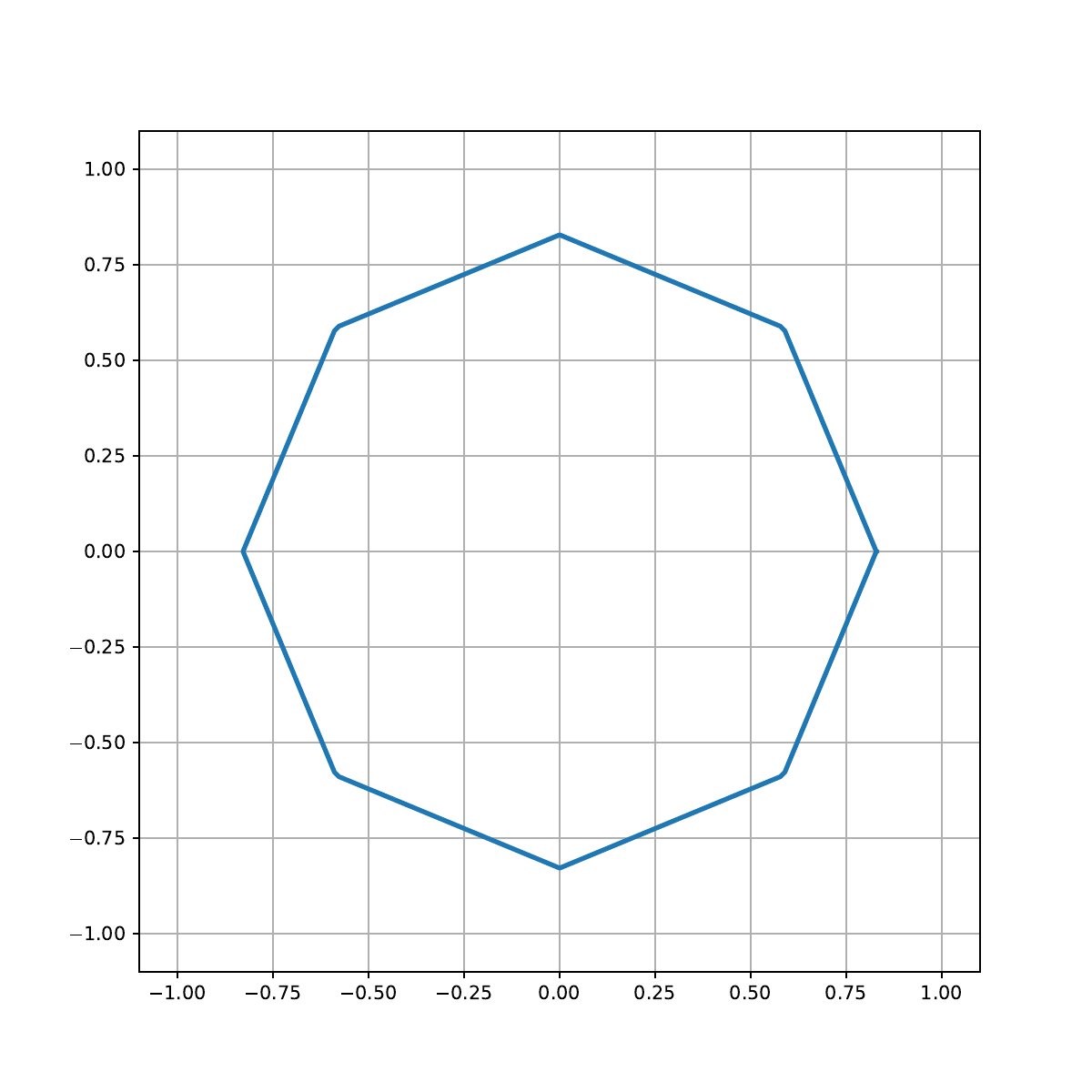}} \\
	\subfloat[][]
	{\includegraphics[width=.43\textwidth, trim={13mm 17mm 20mm 24mm},clip]{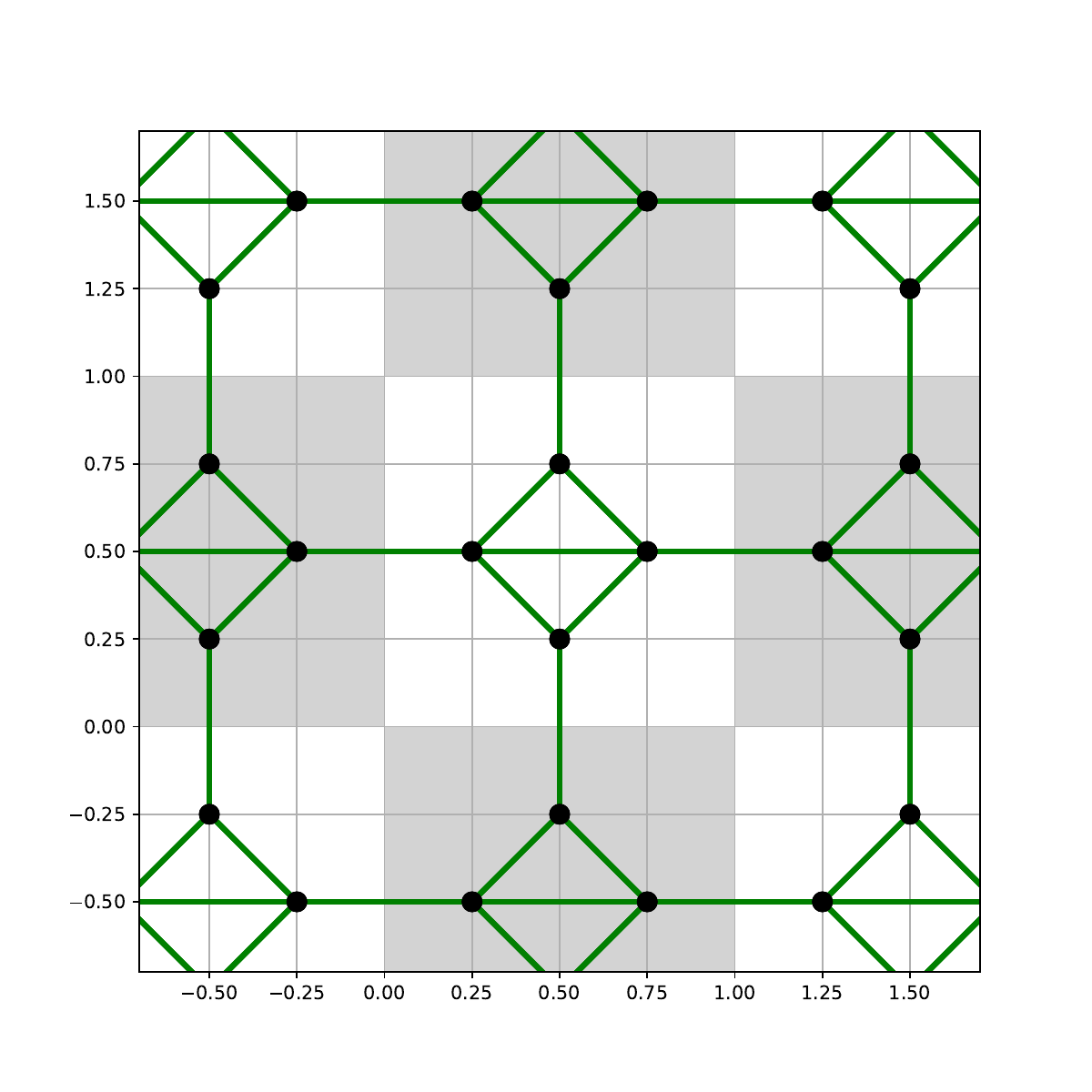}
		\hspace{25pt}
		\includegraphics[width=.43\textwidth, trim={13mm 17mm 20mm 24mm},clip]{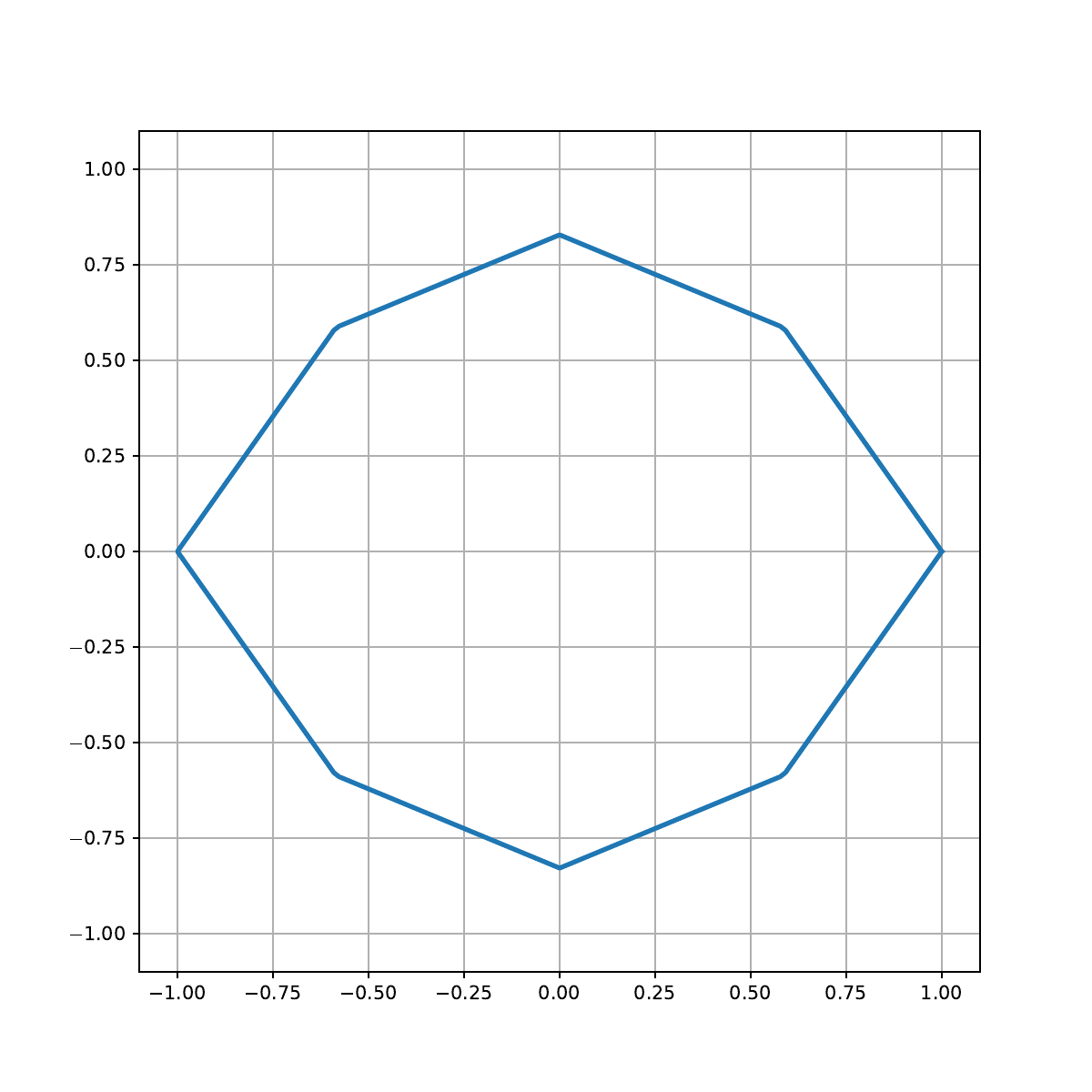}} \\
	\caption{Examples of graphs in~$\R^2$ and corresponding unit balls for~$f_\hom$}
	\label{fig:tilings}
\end{figure}

\clearpage

\newcommand{\etalchar}[1]{$^{#1}$}

\end{document}